\documentclass[11pt]{amsart}

\usepackage[verbose=true]{microtype}
\usepackage{amsmath}
\usepackage{amssymb}
\usepackage{amsthm}
\usepackage{bbm}
\usepackage{mathtools}
\usepackage{color,colortbl}
\usepackage[colorlinks,linkcolor=black,citecolor=black,linktocpage,hypertexnames=true,pdfpagelabels]{hyperref}
\usepackage[all]{xy}
\usepackage{subcaption}
\usepackage{array}
\usepackage{enumerate}
\usepackage{enumitem}
\usepackage{tikz}
\usepackage{relsize}
\usepackage{diagbox}
\usepackage{tikz-cd} 

\usepackage[section]{placeins}

\usepackage{mdframed}
\usepackage[foot]{amsaddr}

\newtheorem{theorem}{Theorem}
\newtheorem{lemma}[theorem]{Lemma}
\newtheorem{corollary}[theorem]{Corollary}
\newtheorem{proposition}[theorem]{Proposition}

\newtheorem{definition}[theorem]{Definition}

\theoremstyle{definition}

\newtheorem{example}[theorem]{Example}
\newtheorem{remark}[theorem]{Remark}
\newtheorem{algorithm}[theorem]{Algorithm}
\newtheorem{method}{Method}

\makeatletter
\renewenvironment{proof}[1][\proofname]{\par
  \normalfont \topsep6\p@\@plus6\p@\relax
  \trivlist
  \item[\hskip\labelsep
        \itshape
    #1\@addpunct{.}]\ignorespaces
}{
  \endtrivlist\@endpefalse
}

\makeatother

\DeclareMathOperator{\trace}{Tr}

\renewcommand{\epsilon}{\varepsilon}

\newcommand{\N}{\mathbb{N}}
\newcommand{\C}{\mathbb{C}}

\newcommand{\diag}{\textrm{diag}}
\newcommand{\rank}{\textrm{rank}}

\usepackage{dsfont}

\newcommand{\be}{\begin{eqnarray}}
\newcommand{\ee}{\end{eqnarray}}
\newcommand{\ben}{\begin{enumerate}}
\newcommand{\een}{\end{enumerate}}

\newcommand{\ba}{\begin{array}}
\newcommand{\ea}{\end{array}}

\newcommand{\mc}{\mathcal}

\newcommand{\tr}{\mathrm{tr}}

\newcommand{\nn}{\nonumber}

\newcommand{\conv}{\mathrm{conv}}

\makeatletter
\def\p@subsection{}
\def\p@subsubsection{}
\makeatother

\let\originalleft\left
\let\originalright\right
\renewcommand{\left}{\mathopen{}\mathclose\bgroup\originalleft}
\renewcommand{\right}{\aftergroup\egroup\originalright}

\newcommand\xqed[1]{
  \leavevmode\unskip\penalty9999 \hbox{}\nobreak\hfill
  \quad\hbox{#1}}
\newcommand\demo{\xqed{$\triangle$}}

\usepackage{centernot,cancel}
\newcommand{\notll}{\centernot{\ll}}

\renewcommand{\arraystretch}{1.7}

\usepackage{multirow}

\begin{document}

\title[Approximate tensor decompositions]{Approximate tensor decompositions: disappearance of many separations}

\author{Gemma De las Cuevas}
\address{Institute for Theoretical Physics, Technikerstr.\ 21a,  A-6020 Innsbruck, Austria}

\author{Andreas Klingler}
\address{Institute for Theoretical Physics, Technikerstr.\ 21a,  A-6020 Innsbruck, Austria}

\author{Tim Netzer}
\address{Department of Mathematics, Technikerstr.\ 13,  A-6020 Innsbruck, Austria}

\date{\today}

\begin{abstract}
It is well-known that tensor decompositions show separations, that is, that 
constraints on local terms (such as positivity) may entail an arbitrarily high cost in their representation.  
Here we show that many of these separations disappear in the approximate case.  
Specifically, for every approximation error $\varepsilon$ and norm, 
we define the approximate rank as the minimum rank of an element in the $\varepsilon$-ball with respect to that norm. 
For positive semidefinite matrices, we show that the separations between rank, purification rank, and separable rank disappear for a large class of Schatten $p$-norms.
For nonnegative tensors, we show  that the separations between rank, positive semidefinite rank, and nonnegative rank disappear for all $\ell_p$-norms with $p>1$. 
For the trace norm ($p = 1$), we obtain upper bounds that depend on the ambient dimension. 
We also provide a deterministic algorithm to obtain the approximate decomposition attaining our bounds.
Our main tool is an approximate version of Carath\'eodory's Theorem. 
Our results imply that many separations are not robust under small perturbations of the tensor, with implications in quantum many-body systems and communication complexity.
\end{abstract}

\maketitle

\tableofcontents

\section{Introduction}

The rank of a matrix is a central notion in mathematics, physics and other sciences.
It is defined as  the minimum number of  rank-$1$ matrices (i.e.\  matrices of the form $v w^t$ where $v,w$ are column vectors)  needed to decompose a matrix as a sum thereof. 
This notion has been extended to include constraints. 
For example, for nonnegative matrices (i.e.\ matrices with nonnegative entries), 
the \emph{nonnegative rank}
is the minimum number of nonnegative rank-$1$ matrices, and the \emph{positive semidefinite rank} 
is the smallest integer $r$ such that the entries of the matrix can be written as $M_{ij} = \textrm{Tr}(E_i F^t_j)$ where $E_i$ and $F_j$ are positive semidefinite matrices of size $r \times r$. 
Similarly, for positive semidefinite matrices, generalisations of these two ranks give rise to the separable rank and the purification rank \cite{De19}. 
For tensors (i.e.\ multilinear maps), there is no unique way of defining `a' rank or a rank with constraints --- instead, there are many meaningful ways of doing so.
The various decompositions and their corresponding ranks can be described by a (weighted) simplicial complex whose vertices are associated to the indices of the tensor \cite{De19d}. 
These decompositions of matrices or tensors with constraints have applications in many areas, including classical and quantum information theory \cite{Fa14,Ja13}, quantum many body systems \cite{Or18}, electrical engineering, data analysis (see \cite{Br10b} and references therein) and machine learning (see \cite{Gl19} and references therein).

Now, it is known that there are \emph{separations} between many of these ranks, that is, one cannot upper bound a rank as a function only depending on another one. For example, there is a separation between the rank and the positive semidefinite rank, and between the latter and the nonnegative rank \cite{Fa14,Go12}, i.e.\ there do not exist functions $f, g: \mathbb{N} \to \mathbb{N}$ such that for all nonnegative matrices $M$ of arbitrary dimension, it holds that 
$$\textrm{psd-rank}(M) \leq f(\textrm{rank}(M)) \quad \textrm{ and } \quad \textrm{nn-rank}(M) \leq g(\textrm{psd-rank}(M)).$$
Throughout the paper, we denote a separation with the symbol  $\ll$, i.e.\ $\textrm{rank} \ll \textrm{psd-rank} \ll \textrm{nn-rank}$. 
Known separations and bounds between these ranks are summarised in Table \ref{tab:RelationsMatrixFact}.

\def\arraystretch{1.4}
\setlength{\tabcolsep}{0pt}

\begin{table}[t]
\begin{tabular}{>{\centering\arraybackslash}p{1.5cm}  r@{\hspace{1mm}} | >{\centering\arraybackslash}p{2cm} >{\centering\arraybackslash}p{2cm} >{\centering\arraybackslash}p{2cm} >{\centering\arraybackslash}p{2cm} }
\multicolumn{2}{c|}{\multirow{2}{*}{$\textrm{rank}_i \quad s_{ij} \quad \textrm{rank}_j$}} & $j=1$ & $j=2$ & $j=3$ & $j=4$ \\
\multicolumn{2}{c|}{} & $\textrm{rank}$ & $\textrm{psd-rank}$ & $\textrm{nn-rank}$ & $\textrm{sqrt-rank}$ \\
\hline
$i=1$& $\textrm{rank}$ & & \cellcolor{gray!30} $\leq (\: \cdot \:)^2$ & \cellcolor{gray!10} $\leq$ & \cellcolor{gray!30} $\leq (\: \cdot \:)^2$\\
$i=2$& $\textrm{psd-rank}$ & \cellcolor{gray!50} $\gg$ & & \cellcolor{gray!10} $\leq$ & \cellcolor{gray!10} $\leq$ \\
$i=3$& $\textrm{nn-rank}$ & \cellcolor{gray!50} $\gg$ & \cellcolor{gray!50} $\gg$ & &  \cellcolor{gray!50} $\gg$\\
$i=4$& $\textrm{sqrt-rank}$ & \cellcolor{gray!50} $\gg$ & \cellcolor{gray!50} $\gg$ & \cellcolor{gray!50} $\gg$ &
\end{tabular}
\caption{Relations between the ranks of different notions of matrix factorizations (see also \cite[Table 1]{Fa14}). $s_{ij}$ denotes the entry in cell $(i,j)$. The same relations hold for quantum versions of these decompositions by subsituting $\textrm{rank}$ by $\textrm{osr}$, $\textrm{psd-rank}$ by $\textrm{puri-rank}$, $\textrm{nn-rank}$ by $\textrm{sep-rank}$, and $\textrm{sqrt-rank}$ by $\textrm{q-sqrt-rank}$ (see \cite{De19} for details).}
\label{tab:RelationsMatrixFact}
\end{table}

These separations imply, for example, that the random communication complexity can be arbitrarily larger than the quantum communication complexity \cite{Ja13}. 
Similar separations hold in the noncommutative setting, namely for the operator Schmidt rank, the purification rank  and the separable rank \cite{De13c,De19,De19d}. Relations between these notions of ranks are again summarized in Table \ref{tab:RelationsMatrixFact}. Here the separations imply that there cannot exist a local transformation from the operator Schmidt decomposition to the local purification form. 
Some of these ranks have already been studied in the approximate case (see \cite{St14} for an upper bound of the approximate positive semidefinite rank with respect to the entrywise maximum-norm or \cite{Ja20} for an upper bound of the approximate operator Schmidt rank for the trace norm).

In this paper, we show that many of these separations disappear in the approximate case.  
This means that many of these separations are not robust to small perturbations of the  matrix or tensor in question.
Our analysis applies to matrices and tensors, possibly with an explicit invariance built-in (see below), and holds for a large class of norms. 

More precisely, we work in the framework of $(\Omega,G)$-decompositions,  introduced in  \cite{De19d}. 
Specifically, we consider elements in tensor product spaces --- especially, we consider (positive semidefinite matrices in) the tensor product of the space of $d\times d$ complex matrices $\mathcal{M}_d$, 
$$
\mathcal{M}_d\otimes \cdots \otimes \mathcal{M}_d, 
$$ 
and (nonnegative tensors in) the space 
$$
\mathbb{C}^d \otimes \cdots \otimes \mathbb{C}^d.
$$
Obviously, each such element can be expressed as a sum of elementary tensor factors, but the summation indices can be arranged in many different ways. We describe the summation indices as facets of a \emph{weighted simplicial complex} (wsc).  
In addition, if the element is invariant under a group action $G$, the decomposition explicitly reflects this invariance, giving rise to a  $(\Omega,G)$-decomposition with an associated $(\Omega,G)$-rank. 

The \emph{approximate $(\Omega,G)$-rank} is defined as the minimal rank of an element 
within an $\varepsilon$-ball of the original element. 
The ball is measured with respect to the Schatten $p$-norm for elements in  $\mc{M}_d\otimes \cdots \otimes\mc{M}_d $, 
and  the $\ell_p$-norm for elements in  $\mathbb{C}^d \otimes \cdots \otimes \mathbb{C}^d$. 
Specifically, for positive semidefinite (psd) matrices, we define the approximate versions of the 
$(\Omega,G)$-rank, 
-purification rank, 
-quantum square root rank, 
and -separable rank. 
For nonnegative tensors, we define the approximate versions of the 
$(\Omega,G)$-rank,
-psd rank, 
-square root rank, 
and -nonnegative rank.

Our main result is that many separations between exact $(\Omega,G)$-ranks disappear in the approximate case. Specifically, for psd matrices, we prove that the separations between the operator Schmidt rank, the purification rank, and the separable rank  disappear in the approximate case for Schatten $p$-norms with $p\in (1,4/3] \cup \{2\}\cup [4,\infty)$  (Corollary \ref{cor:sepMat}). 
For nonnegative tensors, we prove that the separation between rank, nonnegative rank, and psd rank disappear in the approximate case  for $\ell_p$-norms with $p>1$ (Corollary \ref{cor:sepTen}). 

Our central tool is an approximate version of Carath\'eodory's Theorem \cite{Iv19}, which allows to upper bound the approximate rank of an element of a convex set in a dimension independent way. We leverage this  result to prove dimension independent  upper bounds for the approximate $(\Omega,G)$-rank, 
-purification rank 
and -separable rank.
It follows that none of these approximate ranks can diverge, and thus the separations disappear. 

One crucial point is that the approximate Carath\'eodory Theorem needs to be applied to the convex hull of a bounded set. 
Separable states are by definition in the convex hull of all product states, and hence the theorem is directly applicable. 
Yet, more general elements (such as entangled states, or not normalised positive semidefinite matrices) have to be normalized by a \emph{gauge function} in order to be studied with this theorem.  
The gauge function essentially says how much the element needs to rescaled in order to be inside the convex hull of normalized elementary tensors. 
Hence, this disappearance of separations holds up to this rescaling. 
We also show that for entangled states this gauge function is related to the robustness of entanglement \cite{Vi99b}.

We also present an algorithm (Algorithm \ref{alg:Ivanov}) to obtain the approximate decompositions, apply it to several examples, and show how it attains the upper bounds proven in Theorem \ref{Thm:ApproxCaratheodory}.

This paper is organised as follows.
In Section \ref{sec:pre} we the present the approximate Carath\'eodory Theorem for Schatten classes.
In Section \ref{sec:exact} we present several notions of (exact) $(\Omega,G)$-decompositions and their associated ranks, which we define more rigorously in Appendix \ref{sec:app_wsc}. 
In Section \ref{sec:approx} we define and study the approximate $(\Omega,G)$-ranks. 
In Section \ref{sec:disap} we show the disappearance of separations for several approximate ranks.
In Section \ref{sec:algo} we present an algorithm to obtain the approximate decomposition. 
Finally, in Section \ref{sec:concl} we present the conclusions and outlook.

\section{The approximate Carath\'eodory Theorem for Schatten classes}
\label{sec:pre}

In this section we introduce the approximate Carath\'eodory Theorem for Schatten classes, which is a crucial tool in this paper.
Throughout this paper, we denote the set of complex $d\times d$ matrices by $\mathcal{M}_d$. 
We also denote  the unnormalised Schatten $p$-norm by $\Vert\cdot\Vert_p$, i.e.\ for $A\in \mathcal{M}_d$ we have
$$\Vert A\Vert_p \coloneqq \trace\left(|A|^p\right)^{1/p} = \left(\sum_{i=1}^d s_i(A)^p\right)^{1/p},$$
where $|A| \coloneqq \sqrt{A^* A}$, and $A^*$ denotes the complex conjugate transpose of $A$. 
Further, $\{s_i(A)\}_{i=1}^{d}$ denotes the set of singular values of $A$.

The main result of this section is Theorem \ref{thm:approxcara}. 
We start with a version of the approximate Carath\'eodory Theorem \cite{Iv19}, which holds for uniformly smooth Banach spaces.

\begin{definition}
Let $X$ with $\Vert\cdot\Vert$ be a Banach space. The \emph{modulus of smoothness} $\rho_X:[0, \infty] \to [0, \infty]$ is given by 
$$
\rho_X(t) \coloneqq \sup\left\{\frac{1}{2}\left(\Vert x+t y\Vert + \Vert x-t y\Vert\right) - 1 : \Vert x\Vert = \Vert y\Vert \leq 1 \right\} \: \text{ for }\: t \in [0, \infty].
$$ 
A Banach space is called \emph{uniformly smooth} if $\rho_X(t) = o(t)$, i.e.\ ${\rho_X(t)/t \to 0}$ as $t \to 0$.
\end{definition}

Note that a finite dimensional Banach space $X$ with $\Vert  \cdot \Vert$ is uniformly smooth if and only if for all $x,y \in X$ with $\Vert x \Vert = \Vert y \Vert = 1$ the limit
$$ \lim_{t \to 0} \frac{\Vert x + ty\Vert - \Vert x \Vert}{t}$$
exists. In other words, for all $x \neq 0$ all directional derivatives of the norm exist \cite{Me98b}. 
Intuitively, this implies that the unit ball of uniformly smooth spaces is smooth. 

Note that $\rho_X$ is convex, increasing and $\rho_X(0) = 0$. This implies that $\rho_X$ is bijective and we denote the inverse function by $\rho_X^{-1}$.

\begin{theorem}[Approximate Carath\'eodory \cite{Iv19}]
\label{thm:approxcaraIva}
Let $S$ be a bounded set in a uniformly smooth Banach space $X$ equipped with the norm $\Vert\cdot\Vert$,
and
$a \in \mathrm{conv}(S)$. Then there exists a sequence $\{x_i\}_{i=1}^{\infty} \subseteq S$ such that for $a_k = \frac{1}{k} \sum_{i=1}^{k} x_i$ the following inequality holds
$$
\Vert a-a_k\Vert \leq \frac{2 \exp(2)}{k \cdot \rho_X^{-1}(1/k)} \cdot \mathrm{diam}(S)
$$
where $\rho_X(\cdot)$ is the modulus of smoothness.
\end{theorem}

Note that the upper bound of Theorem \ref{thm:approxcaraIva} is independent of $a$, and the only dependence on the space $X$ is given through the diameter of $S$ and the inverse of the modulus of smoothness. We will see that for the special case $X = \mathcal{M}_d$ equipped with the Schatten $p$-norm the upper bound is also dimension independent.

In the following we let $\rho_p(\cdot)$ denote the modulus of smoothness of $\mathcal{M}_d$ with Schatten $p$-norm. In the following we calculate an upper bound for the expression $1/\rho_p^{-1}(1/k)$ following \cite{Iv19,Li63b}.
 A necessary tool for this calculation are \emph{Hanner's inequalities}.

\begin{theorem}[Hanner's inequalities for Schatten norms \cite{Li94}]
\label{thm:hanineq}
Let $A,B \in \mathcal{M}_d$. For $4 \leq p < \infty$ the following inequality holds:
$$
\big(\Vert A\Vert_p + \Vert B\Vert_p \big)^p + 
\big| \Vert A\Vert_p - \Vert B\Vert_p \big|^p 
\geq 
\Vert A+B\Vert_p^p + \Vert A-B\Vert_p^p .
$$
For $1 \leq p \leq 4/3$ the inequality is reversed and for $p = 2$ equality holds.
\end{theorem}

Note that  Theorem \ref{thm:hanineq} is an extension of the original Hanner's inequalities for $\ell_p$-spaces, i.e.\ spaces $X = \mathbb{C}^d$ equipped with the entrywise $p$-norm \cite{Li63b}.
It is widely believed that Theorem \ref{thm:hanineq} is true for $1 < p < \infty$ as in the $\ell_p$ case,  but it is only proven for the given range of $p$. 

\begin{corollary}
The following inequalities hold:
	\[ \rho_{p}(t) \leq \left\{ \begin{array}{ll}
	\frac{1}{p} \cdot t^p &\quad\text{if }1 \leq p \leq 4/3  \\
	\\
	\frac{p-1}{2} \cdot t^2 &\quad\text{if }	p=2  \text{ or } 4 \leq p < \infty
	\end{array}
	\right.\]
	This implies that for $1 < p \leq 4/3$, $p = 2$ and $4 \leq p < \infty$, $\mathcal{M}_d$ equipped with the Schatten $p$-norm is uniformly smooth.
Furthermore, $\mathcal{M}_d$ with the Schatten $1$-norm is not uniformly smooth.
\end{corollary}
\begin{proof}
Let $1 \leq p \leq 4/3$. Applying Hanner's inequality for $A = X+ t Y$ and $B = X- t Y$ we  obtain
\begin{align*}
	\rho_{p}(t) &= \sup\left\{\frac{1}{2}(\Vert X+t Y\Vert_p + \Vert X- t Y\Vert_p) - 1 : \Vert X \Vert_p = \Vert Y \Vert_p = 1 \right\}  \\
	&\leq (1 + t^p)^{1/p} - 1 \leq \frac{t^p}{p}.
\end{align*}
where the last inequality can be obtained by a basic calculation. 

Let $p = 2$ or $4 \leq p < \infty$. Since $x \mapsto |x|^p$ is convex we have for $a,b \geq 0$
$$ 
(a+b)^p \leq 2^{p-1} \cdot (a^p + b^p). 
$$
Using this fact and Hanner's inequality for $A \coloneqq X$ and $B \coloneqq tY$ we obtain
\begin{align*}
\rho_{p}(t) &= \sup\left\{\frac{1}{2}( \Vert X+ t Y\Vert_p + \Vert X-t Y\Vert_p) - 1 : \Vert X\Vert_p = \Vert Y\Vert_p = 1 \right\}   \\
	&\leq \sup\left\{ \left(\frac{\Vert X+tY\Vert_p^p + \Vert X-tY\Vert_p^p}{2}\right)^{1/p} - 1 : \Vert X\Vert_p = \Vert Y\Vert_p = 1 \right\}  \\
	&\leq \left( \frac{(1+t)^p + |1-t|^p}{2} \right)^{1/p} - 1 \leq \frac{p-1}{2} \cdot t^2.
\end{align*}

To prove that the Schatten $1$-class is not uniformly smooth we refer to the $\ell_p$-case (see \cite{Li63b} for details).
\end{proof}

Using the fact that the modulus of smoothness is increasing with $t$, we can formulate the approximate Carath\'eodory Theorem for several Schatten classes.

\begin{theorem}[Approximate Carath\'eodory for Schatten classes]
\label{Thm:ApproxCaratheodory}
Let $\mathcal{S}$ be a bounded set in $\mathcal{M}_d$ equipped with Schatten $p$-norm, where $1 < p \leq 4/3$, $p=2$ or $4 \leq p < \infty$, and $A \in \mathrm{conv}(\mathcal{S})$.
Then there exists a sequence $\{X_i\}_{i=1}^{\infty} \subseteq \mathcal{S}$ such that for ${A_k = \frac{1}{k} \sum_{i=1}^{k} X_i}$ the following inequalities hold:
\begin{enumerate}[label=(\alph*)]
	\item $\displaystyle \Vert A-A_k\Vert_p \leq \frac{2\exp(2)}{p^{1/p}} \cdot k^{1/p - 1} \cdot \mathrm{diam}(\mathcal{S}) \quad \:\text{ if } 1 < p \leq 4/3$
	\item $\displaystyle \Vert A-A_k\Vert_p \leq \exp(2) \cdot \sqrt{\frac{2(p-1)}{k}} \cdot \mathrm{diam}(\mathcal{S}) \: \: \text{if } p=2 \text{ or } 4 \leq p < \infty$
\end{enumerate}
\end{theorem}

\begin{proof}
Using the fact that $\rho_p$ is increasing in its argument, 
and Theorem \ref{thm:approxcaraIva}, we obtain 
$$
\rho_p^{-1}(1/k) \geq \left\{ \begin{array}{ll}\sqrt[p]{\frac{p}{k}} & \text{if } 1 < p \leq 4/3\\
\\
\sqrt{\frac{2}{k \cdot (p-1)}} & \text{if } p=2 \text{ or } 4 \leq p < \infty \\
\end{array}\right.
$$ 
This proves the statement. 
\end{proof}

By fixing some approximation error $\varepsilon > 0$ and $p$, the previous result can be equivalently formulated as an upper bound on the number of summands $k$ necessary to attain an $\varepsilon$-approximation with respect to the Schatten $p$-norm.

\begin{theorem}\label{thm:approxcara}
Let $\mathcal{S} \subseteq \mathcal{M}_d$ bounded, $A\in \textrm{conv}(\mathcal{S})$ and $\epsilon>0$ be given. 
Then in the $\epsilon$-ball with respect to the Schatten $p$-norm around $A$ there is a point $B$ which is a convex combination of at most 
\begin{enumerate}[label=(\alph*)]
	\item $ \displaystyle
\left \lceil C_p \cdot \left(\frac{\textrm{diam}(\mathcal{S})}{\epsilon}\right)^{\frac{p}{p-1}}\right \rceil \quad \:\: \text{if } 1 < p \leq 4/3$
	\item $  \displaystyle
\left \lceil D_p \cdot \left(\frac{\textrm{diam}(\mathcal{S})}{\epsilon}\right)^{2}\right \rceil \quad \quad \text{ if } p = 2  \text{ or } 4 \leq p < \infty
$
\end{enumerate}
points from $S$, where
$$C_p \coloneqq \left(\frac{2 \exp(2)}{p^{1/p}}\right)^{\frac{p}{p-1}} \quad \text{ and } \quad D_p \coloneqq 2 (p-1) \cdot \exp(4).$$
\end{theorem}
\begin{proof}
Follows directly from Theorem \ref{Thm:ApproxCaratheodory}. 
\end{proof}

Note that in both cases (a) and (b) of Theorem \ref{thm:approxcara} the upper bound is dimension independent. The upper bound is best for $p = 2$, and diverges for fixed $\varepsilon > 0$ if we approach $p = 1$. 
The bounds also diverge in the limit $\varepsilon \to 0$ for arbitrary $p$. 
Hence, the approximation error needs to be fixed in order to apply this result. 

\begin{remark}
\label{rem:lp-approxCara}
If we assume that $\mathcal{M}_d$ is equipped with the $\ell_p$-norm $\Vert \cdot \Vert_{\ell_p}$, i.e.\ for $A \in \mathcal{M}_d$ we have
$$ \Vert A \Vert_{\ell_p} \coloneqq \left(\sum_{i,j=1}^{d} |A_{ij}|^p \right)^{1/p},$$
then the upper bounds of Theorem \ref{Thm:ApproxCaratheodory} and Theorem \ref{thm:approxcara} hold for $1 < p \leq 2$ and $p \geq 2$, instead of $1 < p \leq 4/3$ and $4 \leq p < \infty$, respectively. This is due to the fact that Hanner's inequalities are proven for $1 < p < \infty$ in $\ell_p$-spaces \cite{Li63b}. \demo
\end{remark}

\section{Notions of $(\Omega,G)$-decompositions and ranks}
\label{sec:exact}

In this section we present the relevant notions of tensor decompositions based on weighted simplicial complexes  introduced in \cite{De19d}. We will  give a brief overview of the framework of \emph{$(\Omega,G)$-decompositions} and their corresponding \emph{$(\Omega,G)$-ranks} (Section \ref{sec:exdec}), 
as well as their application to psd matrices (Section \ref{sec:psd-ranks}) 
and  nonnegative tensors (Section \ref{sec:tensorranks}).
The technical parts of this section are not essential for the main results of this paper, which are presented in Section \ref{sec:approx}. 
Only the definitions of exact $(\Omega,G)$-ranks will be necessary for the subsequent definitions of approximate $(\Omega,G)$-ranks. 

\subsection{Exact decompositions}
\label{sec:exdec}

The goal of this section is to explain the notions of weighted simplicial complexes (wsc) $\Omega$ and group actions $G$ on $\Omega$ to finally define an $(\Omega,G)$-decomposition (Definition \ref{def:omegaG-dec}). 
In order to motivate them, we will first give an intuitive explanation of these concepts and then illustrate them with Example \ref{ex:decompositions1}.
For a rigorous definition of wsc and group action we refer to Appendix \ref{sec:app_wsc} and \cite{De19d}.

For each index $i \in [n] \coloneqq \{0, \ldots, n\}$, we fix a $\mathbb{C}$-vector space $\mathcal{V}_i$ (called the \emph{local vector space}).
We denote the \emph{global vector space} as the tensor product space
$$ \mathcal{V} \coloneqq \mathcal{V}_0 \otimes \cdots \otimes \mathcal{V}_n.$$
By definition every $v \in \mathcal{V}$ can be expressed as a sum of elementary tensors
$$
v^{[0]} \otimes \cdots \otimes v^{[n]}.
$$
The different ways of arranging the summation indices will be reflected in the wsc --- specifically, the summation indices will be associated to the facets of the wsc.
A weighted simplicial complex is intuitively a `well-formed' multi-hypergraph. 
A multi-hypergraph is generalization of a graph where connections of possibly more than two vertices are allowed, thus giving rise to a hypergraph. In addition, every facet can be repeated, i.e.\ there can be multiple copies thereof --- this gives rise to a \emph{multi}-hypergraph.

Furthermore, if $v$ is invariant with respect to permutations of the indices $[n]$ of  elementary tensors, we are interested in obtaining explicitly invariant decompositions, i.e.\ decompositions whose elementary tensors themselves are invariant under the permutation of the indices $[n]$. 
Thus, for a given group action $G$ on the set $[n]$, we consider the induced linear group action on $\mathcal{V}$, i.e.\
$$ g: v^{[0]} \otimes \cdots \otimes v^{[n]} \mapsto v^{[g0]} \otimes \cdots \otimes v^{[gn]}$$
for $g \in G$. An element $v \in \mathcal{V}$ is called \emph{$G$-invariant} if it is invariant under the action of $G$ on $\mathcal V$. The subspace of invariant elements is denoted $\mathcal{V}_\textrm{inv}$.

\begin{example}
\label{ex:decompositions1}
\begin{enumerate}[label=(\roman*), wide=\parindent, labelwidth=!]
	\item Consider a decomposition of the form
	$$v = \sum_{\alpha = 1}^{r} v_{\alpha}^{[0]} \otimes \cdots \otimes v_{\alpha}^{[n]}.$$
	The minimal number $r$ of elementary tensors is called the \emph{tensor rank} of $v$.
	Setting $\mathcal{I} = \{1, \ldots, r\}$ and $\mathcal{F} = \{\{0, \ldots, n\}\}$ we can equivalently write the decomposition as a sum over all functions $\alpha: \mathcal{F} \to \mathcal{I}$, denoted $\alpha \in \mathcal{I}^{\mathcal{F}}$, i.e.\
	\begin{equation}
	 \label{eq:exdec}
	v = \sum_{\alpha \in \mathcal{I}^{\mathcal{F}}} v_{\alpha_{|_0}}^{[0]} \otimes \cdots \otimes v_{\alpha_{|_n}}^{[n]},
	\end{equation}
	where $\alpha_{|_i}$ denotes the restriction of $\alpha$ to the set $\mathcal{F}_i = \{F \in \mathcal{F}: i \in F\}$. Intuitively, every element $F \in \mathcal{F}$ corresponds to a summation index $\alpha_i$, and the elements of $F$ correspond to the positions where $\alpha_i$ appears.
	\item Let $G$ be a transitive group action on $[n]$, i.e.\ there is only one orbit, and hence $Gi = [n]$ for all $i \in [n]$. 
	For elements $v$ which are invariant under the group action $G$, we consider the $G$-invariant decomposition
	$$v = \sum_{\alpha = 1}^{r} v_{\alpha} \otimes \cdots \otimes v_{\alpha}$$
	which is called the \emph{symmetric tensor decomposition}.
	The smallest number $r$ among all possible decompositions is called the \emph{symmetric tensor rank} (see Example \ref{ex:wsc} (i) for an explicit example with $n=2$).
	\item Consider a decomposition of the form
	$$v = \sum_{\alpha_0, \ldots, \alpha_{n-1} = 1}^{r} v_{\alpha_0}^{[0]} \otimes v_{\alpha_0, \alpha_1}^{[1]} \otimes \cdots \otimes v_{\alpha_{n-2}, \alpha_{n-1}}^{[n-1]}\otimes v_{\alpha_{n-1}}^{[n]}.
	$$
	This decomposition is called the \emph{matrix product operator form} and the minimal $r$ among all possible decompositions the \emph{operator Schmidt rank}. 
	This corresponds to the decomposition of Equation \eqref{eq:exdec} with $$\mathcal{F} = \{\{0,1\}, \{1,2\}, \ldots, \{n-1, n\}\}$$ and ${\mathcal{I} = \{1, \ldots, r\}}$.
 	\item Consider a decomposition of the form
 	$$v = \sum_{\alpha_0, \ldots, \alpha_{n-1}=1}^{r} v_{\alpha_0, \alpha_1}^{[0]} \otimes v_{\alpha_1, \alpha_2}^{[1]} \otimes \cdots \otimes v_{\alpha_{n-2}, \alpha_{n-1}}^{[n-2]}\otimes v_{\alpha_{n-1}, \alpha_0}^{[n-1]}.$$
	This is similar to (iii), but with periodic boundary conditions. 
	This corresponds to choosing
 	$$\mathcal{F} = \{\{i, i+1\}: i \in \{0, \ldots, n-2\}\} \cup \{\{n-1, 0\}\} $$ 
	in  Equation \eqref{eq:exdec}.
	
	Additionally considering the symmetry operation given by the cyclic group ${G = C_n}$
	(i.e.\ the group generated by a mapping $i \mapsto i+1$, where addition is modulo $n$) we obtain the \emph{translational invariant matrix product operator form} \cite{De19}
	$$v = \sum_{\alpha_0, \ldots, \alpha_{n-1} = 1}^{r} v_{\alpha_0, \alpha_1} \otimes v_{\alpha_1, \alpha_2} \otimes \cdots \otimes v_{\alpha_{n-2}, \alpha_{n-1}}\otimes v_{\alpha_{n-1}, \alpha_0}.$$
\end{enumerate}
\demo
\end{example}

Example \ref{ex:decompositions1} shows a unified framework for various decompositions. 
In the following, we give examples of simplicial complexes, and relate their facets to the  sets $\mathcal{F}$ constructed in Example \ref{ex:decompositions1}. Since facets can appear multiple times we will denote the multiset of all facets by $\widetilde{\mathcal{F}}$.

\begin{example}\label{ex:wsc}
(i) The simplicial complex (sc) $\Sigma_n$ which contains all possible connections between vertices is called the \emph{$n$-simplex}. For $n=2$ this can be depicted as follows:
\bigskip
\begin{center}
\begin{tikzpicture}
\filldraw (-1,0) circle (2pt);
\filldraw (1,0) circle (2pt);
\filldraw (0,1.5) circle (2pt);
\draw[thick] (-1,0) -- (1,0) -- (0,1.5)-- cycle;
\filldraw[opacity=0.4] (-1,0) -- (1,0) -- (0,1.5);
\put (-40,-5) {$0$};
\put (35,-5) {$1$};
\put (-2,48) {$2$};
\end{tikzpicture}
\end{center}

\bigskip\noindent
Note that the gray area shows the facet which connects all three vertices.
$\Sigma_n$ has only one (multi-)facet, i.e.\ $\mathcal F=\widetilde{\mathcal F}=\{[n]\}$ and hence gives rise to a tensor decomposition shown in (i) or (ii) of Example \ref{ex:decompositions1}.

(ii) For $n\geq 1$, the \emph{line of length $n$} (i.e.\ composed of $n+1$ points) is the sc $\Lambda_n$ corresponding to the following graph:

\bigskip
\begin{center}
\begin{tikzpicture}
\filldraw (-2,0) circle (2pt);\put (-60,-15){$0$};
\filldraw (-1,0) circle (2pt);\put (-31,-15){$1$};
\filldraw (0,0) circle (2pt);\put (-2,-15){$2$};
\put (20,-2.5) {$\cdots$};
\filldraw (2,0) circle (2pt);\put (55,-15){$n$};
\draw[thick] (-2,0) -- (-1,0);
\draw[thick] (-1,0) -- (0,0);
\draw[thick] (0,0) -- (0.4,0);
\draw[thick] (1.5,0) -- (2,0);
\end{tikzpicture}
\end{center}

\bigskip\noindent The set $\mathcal F=\widetilde{\mathcal F}$ has $n$ elements and generates a matrix product operator form given in (iii) of Example \ref{ex:decompositions1}. 
Intuitively the sc describes the connections between the local vector spaces through shared summation indices.

(iii) For $n\geq 3$, the \emph{circle of length $n$} is the sc $\Theta_n$  corresponding to  the following graph:

\bigskip
\begin{center}
\begin{tikzpicture}
\filldraw (-1,0) circle (2pt);\put (-40,-3) {$0$};
\filldraw (-0.707,0.707) circle (2pt);\put (-30,20) {$1$};
\filldraw (0,1) circle (2pt);\put (-3,35) {$2$};
\filldraw (0.707,0.707) circle (2pt);\put (30,20) {$3$};
\filldraw (1,0) circle (2pt);\put (40,-3) {$4$};
\filldraw (0,-1) circle (2pt);\put (-13,-38) {$n-2$};
\filldraw (-0.707,-0.707) circle (2pt);\put (-50,-23) {$n-1$};
\draw[thick](-1,0)--(-0.707,0.707) -- (0,1)--(0.707,0.707) -- (1,0)-- (0.9,-0.5);
\draw[thick](0.4,-0.9)--  (0,-1)-- (-0.707,-0.707)--(-1,0) ;
\put(13,-23) {\rotatebox[origin=c]{35}{$\cdots$}};
\end{tikzpicture}
\end{center}

\bigskip\noindent It has $n$ facets and generates the decomposition given in Example \ref{ex:decompositions1} (iv).
\demo
\end{example}

A group action of $G$ on the wsc $\Omega$ is a natural extension of the group action introduced at the beginning of Section \ref{sec:exdec}. It is defined on the power set of $[n]$ and maps subsets
$$
[n] \supseteq \{i_1, i_2, \ldots, i_k\} \mapsto \{gi_1, gi_2, \ldots, gi_k\}.
$$
Additionally it respects the structure of the wsc, i.e.\ facets in $\mathcal{\widetilde{F}}$ are mapped to facets in $\mathcal{\widetilde{F}}$.
For a rigorous definition of this notion we refer to Appendix \ref{sec:app_wsc} and \cite{De19d}; for examples we refer to Example \ref{ex:decompositions1}.
Throughout this paper, we assume that every group action $G$ is a valid 
group action on $\Omega$.

Furthermore, we call a group action $G$ on $\Omega$ \emph{free} if for all $F \in \mathcal{\widetilde{F}}$ the only element $g \in G$ mapping $F$ to itself is the neutral element.

Finally, we define the notion of $(\Omega,G)$-decomposition and the $(\Omega,G)$-rank.
 For two sets $X,Y$, the set $Y^X$ contains, by definition, all functions $f:X \to Y$. For any set $\mathcal{I}$, any $\alpha \in \mathcal{I}^{\widetilde{\mathcal{F}}}$ and any $i \in \mathcal{I}$, we  call the restriction to $\widetilde{\mathcal{F}}_i$
$$\alpha_{|_{\widetilde{\mathcal{F}}_i}} \in \mathcal{I}^{\widetilde{\mathcal{F}}_i}$$
the \emph{restriction of $\alpha$ to vertex $i$}, and write $\mathcal{\alpha}_{|i}$ instead.

\begin{definition}
\label{def:omegaG-dec}
For $v \in \mathcal{V}$, an \emph{$(\Omega,G)$-decomposition} is given by a finite index set $\mathcal{I}$ and families
$$ V^{[i]} \coloneqq \left(v_{\beta}^{[i]}\right)_{\beta \in \mathcal{I}^{\widetilde{\mathcal{F}}_i}},$$
where $v_{\beta}^{[i]} \in \mathcal{V}_i$ for $i \in [n]$ such that:
\begin{enumerate}[label=(\alph*)]
	\item We have
	\begin{equation}\label{eq:OmegaG-dec}
	v = \sum_{\alpha \in \mathcal{I}^{\widetilde{\mathcal{F}}}} v_{\alpha_{|0}}^{[0]} \otimes v_{\alpha_{|1}}^{[1]} \otimes \cdots \otimes v_{\alpha_{|n}}^{[n]}.
	\end{equation}
	\item For all $i \in [n]$, $g \in G$ and $\beta \in \mathcal{I}^{\widetilde{\mathcal{F}}_i}$ it holds
	$$ v_\beta^{[i]} = v_{{}^g\beta}^{[gi]},$$
	where ${}^g \beta \in \mathcal{I}^{\widetilde{\mathcal{F}}_{gi}}$ and ${}^g\beta(F) \coloneqq \beta(g^{-1}F)$ for $F \in \widetilde{\mathcal{F}}_{gi}$.
\end{enumerate}
The smallest cardinality of $\mathcal{I}$ among all possible $(\Omega,G)$-decompositions of $v$ is called the \emph{$(\Omega,G)$-rank of $v$}, denoted 
$$\emph{\textrm{rank}}_{(\Omega,G)}(v).$$
If no $(\Omega,G)$-decomposition of $v$ exists, we set $\textrm{rank}_{(\Omega,G)}(v) \coloneqq \infty$.
\end{definition}

Note that if $G$ is a free group action on a connected wsc $\Omega$, and $v \in \mathcal{V}_\text{inv}$, there always exists an $(\Omega,G)$-decomposition of $v$  \cite[Thm. 13]{De19d}. 

For simplicity, if $G$ is the trivial group, we will call an $(\Omega,G)$-decom\-position just an $\Omega$-decomposition, and write $\textrm{rank}_{\Omega}$ for the rank. 
The same simplification will also be used for all ranks defined in the following two subsections.

\subsection{Exact ranks for psd matrices}
\label{sec:psd-ranks}
Separability (or its negation, entanglement), and purifications are central notions in quantum information theory. In the next two definitions we will formulate these notions in the framework of $(\Omega,G)$-decompositions.

Throughout this section we will fix a connected wsc $\Omega$ together with a free group action $G$. 
We will also assume that the local vector space is given by
$$ \mathcal{V}_i \coloneqq \mathcal{M}_{d_i}$$
and hence
$$ \mathcal{V} \coloneqq \mathcal{M}_{d_0} \otimes \cdots \otimes \mathcal{M}_{d_n} \cong \mathcal{M}_{d_0 \cdots d_n},$$
whose hermitian part is ${\rm Her}_{d_0}\otimes\cdots\otimes{\rm Her}_{d_n}\cong{\rm Her}_{d_0\cdots d_n}$.
The $d_i$ can be chosen differently as long as there are no further restrictions given by the group action of $G$ on $\mathcal{V}$; that is, whenever $i,j$ are in the same orbit of the group action of $G$ on $[n]$, then $d_i = d_j$.
Further we will define the \emph{cone of (complex) psd matrices} as $\mathcal{M}^{+}_{d_0 \cdots d_n}$.
If $\rho \in \mathcal{M}_{d_0 \cdots d_n}^{+}$ fulfills $\mathrm{Tr}(\rho) = 1$, we call it a \emph{state}. 

Let us now define $(\Omega,G)$-purifications and $(\Omega,G)$-square root decompositions.

\begin{definition}
\label{def:puridec}
Let $\rho \in \mathcal{M}_{d_0 \cdots d_n}^{+}$.
\begin{enumerate}[label=(\roman*), wide, labelwidth=!, labelindent=0pt]
	\item An \emph{$(\Omega,G)$-purification} is an element $${\sigma \in \mathcal{M}_{d'_0, d_0} \otimes \cdots \otimes \mathcal{M}_{d'_n, d_n}}$$ with

	$$ \rho = \sigma^* \sigma \: \text{ and } \: \textrm{rank}_{(\Omega,G)}(\sigma) < \infty,$$	where $\mathcal{M}_{d_i, d'_i}$ denotes the space of all complex $d'_i \times d_i$ matrices and ${}^*$ the adjoint. 
	The smallest $(\Omega,G)$-rank among all $(\Omega,G)$-purifications is called \emph{$(\Omega,G)$-purification rank of $\rho$}, denoted
	\emph{$$\textrm{puri-rank}_{(\Omega,G)}(\rho).$$}
	\item $\sigma \in \textrm{Her}_{d_0} \otimes \cdots \otimes \textrm{Her}_{d_n}$ is called \emph{square root of $\rho$} if $\sigma^2 = \rho$.
	The smallest $(\Omega,G)$-rank among all square roots of $\rho$ is called \emph{$(\Omega,G)$-quantum square root rank of $\rho$}
	, denoted
	\emph{$$\textrm{q-sqrt-rank}_{(\Omega,G)}(\rho).$$}
	\end{enumerate}
\end{definition}
\begin{remark}
\label{rem:purisqrtrk}
(i) Note that every square root is a purification and hence
$$\textrm{puri-rank}_{(\Omega,G)}(\rho) \leq \textrm{q-sqrt-rank}_{(\Omega,G)}(\rho).$$

(ii) Note that not every matrix $\sigma$ which fulfills $\sigma^2 = \rho$ is automatically hermitian.
From the spectral decomposition,  $\rho = U D U^*$  
 with the diagonal matrix ${D = \textrm{diag}(\lambda_1, \lambda_2, \ldots)}$, 
we see that all hermitian square roots are of the form \cite{Hi08}
$$ \sigma = U D^{1/2} U^*, \quad D^{1/2} = \textrm{diag}\left(\pm \sqrt{\lambda_1}, \pm \sqrt{\lambda_2}, \ldots \right).
$$

(iii) If $G$ is a free group action on a connected wsc $\Omega$ and $\rho \in \mathcal{V}_{\text{inv}}$ and is psd, there always exists an $(\Omega,G)$-purification of $\rho$ and a square root with finite $(\Omega,G)$-rank \cite[Thm. 27]{De19d}.\demo
\end{remark}

The next step is the definition of the separable $(\Omega,G)$-rank.
We call the matrix ${\rho \in \mathcal{M}_{d_0 \cdots d_n}^{+}}$ \emph{separable} if it admits a decomposition
$$
 \rho = \sum_{j} \rho^{[0]}_j \otimes \cdots \otimes \rho^{[n]}_j
 $$
where $\rho^{[i]}_j \in \mathcal{M}_{d_i}^{+}$. If additionally $\mathrm{Tr}(\rho) = 1$, we call $\rho$ a \emph{separable state}. From now on, we will denote the set of separable states
$$ \mathrm{SEP}_{d_0, d_1, \ldots, d_n} \coloneqq \{\rho \in \mathcal{M}_{d_0 \cdots d_n}^{+}: \rho \text{ separable state}\}.$$
If $d_i = d$ for all $i \in [n]$, we will write for simplicity $\mathrm{SEP}_{n,d} \coloneqq \mathrm{SEP}_{d_0, \ldots, d_n}$.\\
\begin{definition}
\label{def:sepdec}
A \emph{separable $(\Omega,G)$-decomposition of $\rho \in \mathcal{M}_{d_0 \cdots d_n}$} is given by an $(\Omega,G)$-decomposition $$\rho = \sum_{\alpha \in \mathcal{I}^{\widetilde{\mathcal{F}}}} \rho_{\alpha_{|0}}^{[0]} \otimes \rho_{\alpha_{|1}}^{[1]} \otimes \cdots \otimes \rho_{\alpha_{|n}}^{[n]}$$ in which $\rho_{\beta}^{[i]} \in \mathcal{M}_{d_i}^{+}$ for $\beta \in \mathcal{I}^{\widetilde{\mathcal{F}}_i}$ and $i \in [n]$. 
The smallest cardinality of an index set $\mathcal{I}$ among all possible separable $(\Omega,G)$-decompositions of $\rho$ is called the \emph{separable $(\Omega,G)$-rank of $\rho$}, denoted $$\emph{\textrm{sep-rank}}_{(\Omega,G)}(\rho).$$
If there exists no separable $(\Omega,G)$-decomposition of $v$, we set $\textrm{sep-rank}_{(\Omega,G)}(v)$ to $\infty$.
\end{definition}
Note that if $G$ is a free group action on a connected wsc $\Omega$ and $\rho \in \mathcal{V}_\text{inv}$ is separable, there always exists a separable $(\Omega,G)$-decomposition of $\rho$  \cite[Thm. 21]{De19d}. 

\subsection{Exact ranks for nonnegative tensors}
\label{sec:tensorranks}
In the following we will consider the set of nonnegative tensors,  and will define different notions of $(\Omega,G)$-ranks based on \cite[Sec.\ 5]{De19d}.

For simplicity we consider the local space $\mathcal{V}_i = \mathbb{C}^d$ and define the global space $\mathcal{V} \coloneqq \mathcal{K}_{n,d}$, where
$$\mathcal{K}_{n,d} \coloneqq \bigotimes_{i=0}^{n} \mathbb{C}^{d}.$$
If $n$ and $d$ are clear from the context, we write $\mathcal{K}$ instead of $\mathcal{K}_{n,d}$. Any element $M \in \mathcal{K}$ can be uniquely written as 
$$ M = \sum_{i_0, \ldots, i_n} m_{i_0, \ldots, i_n} e_{i_0} \otimes \cdots \otimes e_{i_n}$$
where $e_j$ denotes the $j$-th standard basis vector in the corresponding vector space $\mathbb{C}^{d}$.
$M$ is said to be \textit{nonnegative} if $m_{i_0, \ldots, i_n} \geq 0$ for all $i_0, \ldots, i_n$. 

Recall that every tensor $M \in \mathcal{K}_{n,d}$ can be associated with a diagonal matrix $\sigma \in \mathcal{M}_d \otimes \cdots \otimes \mathcal{M}_d \cong \mathcal{M}_{d^{n+1}}$ by setting
\begin{equation}
\label{eq:correspondence}
\sigma = \sum_{i_0 \ldots, i_n} m_{i_0, \ldots, i_n} E_{i_0 i_0} \otimes \cdots \otimes E_{i_n i_n},
\end{equation}
where $E_{jk}$ is the matrix which has value $1$ on position $(j,k)$ and $0$ elsewhere. Obviously $\sigma$ is psd if and only if $M$ is nonnegative. 

We now give a brief description
of several $(\Omega,G)$-decompositions of nonnegative tensors and their corresponding ranks. For a more detailed discussion we refer to \cite{De19d}.

\begin{definition}
\label{def:tensordec}
\begin{enumerate}[label=(\roman*), wide, labelwidth=!, labelindent=0pt]
	\item A \emph{nonnegative $(\Omega,G)$-decomposition of $M \in \mathcal{K}_{n,d}$} is an $(\Omega,G)$-decom\-position as in Equation \eqref{eq:OmegaG-dec} where all $v_{\alpha_{|i}}^{[i]} \in \C^d$ have nonnegative entries. 
	The corresponding rank is called the \emph{nonnegative $(\Omega,G)$-rank of $M$}, denoted
	\emph{$$\textrm{nn-rank}_{(\Omega,G)}(M).$$}
	\item A \emph{positive semidefinite $(\Omega,G)$-decomposition of $M \in \mathcal{K}_{n,d}$} consists of psd matrices
	$$E_j^{[i]} \in \mathcal{M}_{k_i}^{+}$$ 
	where $k_i = \left|\mathcal{I}^{\widetilde{\mathcal{F}}_i}\right|$ for $i \in [n]$ and $j \in \{1, \ldots, d\}$, such that
	$$\left( E_j^{[i]} \right)_{\beta, \beta'} = \left( E_j^{[gi]} \right)_{{}^g \beta, {}^g \beta'}$$
	for all $i,g,j, \beta, \beta'$, where ${}^g\beta(F) \coloneqq \beta(g^{-1}F)$ for $F \in \widetilde{\mathcal{F}}_{gi}$, and
	$$m_{i_0, \ldots, i_n} = \sum_{\alpha, \alpha' \in \mathcal{I}^{\widetilde{\mathcal{F}}}} \left (E_{i_0}^{[0]}\right)_{\alpha_{|0}, \alpha'_{|0}} \cdots \left(E_{i_n}^{[n]} \right)_{\alpha_{|n}, \alpha'_{|n}}$$
	for all $i_0, \ldots, i_n$. 
	The smallest cardinality of an index set $\mathcal{I}$ among all possible positive semidefinite $(\Omega,G)$-decompositions is called the \emph{positive semidefinite $(\Omega,G)$-rank of $M$}, denoted
	\emph{$$\textrm{psd-rank}_{(\Omega,G)}(M).$$}
	\item $N \in \mathcal{K}_{n,d}$ with $n_{i_0, \ldots, i_n} \in \mathbb{R}$ is called a \emph{square root of $M$}, if $M = N \circ N$, where $\circ $ denotes the Hadamard product (i.e.\ entrywise multiplication, $m^{}_{i_0, \ldots, i_n} = n^2_{i_0, \ldots, i_n}$). 
	The smallest $(\Omega,G)$-rank among all square roots of $M$ is called \emph{$(\Omega,G)$-square root rank of $M$}, denoted
	\emph{$$\textrm{sqrt-rank}_{(\Omega,G)}(M).$$}
\end{enumerate}
\end{definition}

\begin{remark}
\label{rem:rank-relations}
For the psd matrix $\sigma$ and the nonnegative tensor $M$ of
Equation \eqref{eq:correspondence} the following relations hold (see \cite[Thm. 43]{De19d} for details):
\begin{enumerate}[label=(\alph*)]
	\item $\textrm{rank}_{(\Omega,G)}(M) = \textrm{rank}_{(\Omega,G)}(\sigma)$
	\item $\textrm{nn-rank}_{(\Omega,G)}(M) = \textrm{sep-rank}_{(\Omega,G)}(\sigma)$
	\item $\textrm{psd-rank}_{(\Omega,G)}(M) = \textrm{puri-rank}_{(\Omega,G)}(\sigma)$
	\item $\textrm{sqrt-rank}_{(\Omega,G)}(M) = \textrm{q-sqrt-rank}_{(\Omega,G)}(\sigma)$.
\end{enumerate}
\demo
\end{remark}

\section{Approximate $(\Omega,G)$-decompositions and ranks}
\label{sec:approx}

In this section we will define the notions of approximate $(\Omega,G)$-decomp\-ositions for both psd matrices and nonnegative tensors, and will apply the results from Section \ref{sec:pre} to obtain upper bounds for ranks of approximate $(\Omega,G)$-decompositions.

The section is organized as follows. 
In Section \ref{sec:approx-psd-ranks} and Section \ref{sec:approx-tensorranks} we define the approximate analogues of the $(\Omega,G)$-ranks for psd matrices and nonnegative tensors, respectively.
In Section \ref{sec:norms} we introduce gauge functions,  a relevant tool to obtain the upper bounds. 
Subsequently, we show upper bounds for general matrices (Section \ref{ssec:general}), 
psd matrices (Section \ref{ssec:psd}) 
and separable states (Section \ref{ssec:sep}). 

As before we consider 
$$\mathcal{M}_{d_0} \otimes \mathcal{M}_{d_1} \otimes \cdots \otimes\mathcal{M}_{d_n}\cong\mathcal{M}_{d_0\cdots d_n}.$$
We further fix a connected wsc $\Omega$ and a free group action $G$ on $\Omega$.

\subsection{Approximate decompositions of psd matrices}
\label{sec:approx-psd-ranks}

We now introduce the different notions of approximate $(\Omega,G)$-ranks in the space $\mathcal{M}_{d_0} \otimes \cdots \otimes \mathcal{M}_{d_n}$. 
Generally speaking, given a matrix $\rho$, the approximate rank is the minimal rank of all matrices contained in the $\varepsilon$-ball of $\rho$ with respect to the Schatten $p$-norm $\Vert \cdot \Vert_p$.
Note that this is different  to the notion of \emph{border rank}, which is the minimial $k \in \mathbb{N}$ among all possible sequences $(\rho_n)_{n \in \mathbb{N}}$ consisting of rank-$k$ matrices $\rho_n$ converging to $\rho$ \cite{Br19}.

\begin{definition} \label{def:approxranks}
Let $p\in [1,\infty)$ and $\varepsilon > 0$. Further let $M \in \mathcal{M}_{d_0} \otimes \cdots \otimes \mathcal{M}_{d_n}$ and ${\rho \in \mathcal{M}_{d_0 \cdots d_n}^{+}}$.
We define 
$$\textrm{rank}^{\epsilon,p}_{(\Omega,G)}(M) \coloneqq \min \{\textrm{rank}_{(\Omega,G)}(N) : \Vert M-N\Vert_p\leq \epsilon, N \in \mathcal{M}_{d_0 \cdots d_n}\} ,$$ and similarly 
$$\textrm{puri-rank}^{\epsilon,p}_{(\Omega,G)}(\rho) \coloneqq \min \{\textrm{puri-rank}_{(\Omega,G)}(\sigma) : \Vert \rho-\sigma\Vert_p\leq \epsilon, \sigma\in \mathcal{M}_{d_0 \cdots d_n}\},$$
$$\textrm{q-sqrt-rank}_{(\Omega,G)}^{\epsilon, p}(\rho) \coloneqq \min\{\textrm{q-sqrt-rank}_{(\Omega, G)}(\sigma) : \Vert\rho - \sigma\Vert_p \leq \varepsilon, \sigma \in \mathcal{M}_{d_0 \cdots d_n}\}$$
and
$$
\textrm{sep-rank}^{\epsilon,p}_{(\Omega,G)}(\rho) \coloneqq \min \{\textrm{sep-rank}_{(\Omega,G)}(\sigma) : \Vert\rho-\sigma\Vert_p\leq \epsilon, \sigma \in \mathcal{M}_{d_0 \cdots d_n}\}. 
$$
\end{definition}

Recall that a non-psd matrix $\rho$ does not admit a purification, and thus its purification rank is $\infty$. 
In contrast, $\textrm{puri-rank}^{\varepsilon,p}_{(\Omega,G)}(\rho)$ might be finite ---  it is finite if and only if there is a psd matrix in the $\varepsilon$-ball of $\rho$ with respect to the Schatten $p$-norm.
Similar statements hold for the quantum square root rank and the separable rank. 

Let us now revisit Example \ref{ex:decompositions1} and Example \ref{ex:wsc}, and explain the notions of approximate ranks in the cases therein.

\begin{example}
\label{ex:approxdecompositions}
Let $M \in \mathcal{M}_{d_0 \cdots d_n}$, $\varepsilon > 0$ and $p \in [1,\infty)$.
\begin{enumerate}[label=(\roman*), wide=\parindent, labelwidth=!]
	\item An approximate $\Sigma_n$-decomposition is given by a matrix $N \in \mathcal{M}_{d_0 \cdots d_n}$ attaining a decomposition

	$$ N = \sum_{\alpha = 1}^{r} N_{\alpha}^{[0]} \otimes \cdots \otimes N_{\alpha}^{[n]}$$ 
		with $\Vert M-N\Vert_p \leq \varepsilon$. 
	The approximate $\Sigma_n$-rank, $\textrm{rank}_{\Sigma_n}^{\varepsilon, p}(\rho)$, is called \emph{approximate tensor rank} of $M$ and is the smallest integer $r$ among all matrices $N$ in the $\varepsilon$-ball of $M$.
	\item Consider the line $\Omega = \Lambda_n$ of length $n$. The \emph{approximate operator Schmidt rank}, $\textrm{rank}_{\Lambda_n}^{\varepsilon,p}(M)$, is the minimal integer $r$ among all $N \in \mathcal{M}_{d_0 \cdots d_n}$ with $\Vert M-N\Vert_p \leq \varepsilon$ and all decompositions of  matrix product operator form.
	\item For $n\geq 3$ consider the circle $\Omega = \Theta_n$ of length $n$ together with the cyclic group $G = C_n$ whose elements are translations of the points on the line. Further let $M \in \mathcal{M}_{d^n}$ be $C_n$-invariant. In this example $C_n$-invariance corresponds to translational invariance of $M$. The \emph{approximate translational invariant operator Schmidt rank} of $M$, $\textrm{rank}_{(\Theta_n,C_n)}^{\varepsilon, p}(M)$, is the minimal integer $r$ among all $N \in \mathcal{M}_{d^n}$ with $\Vert M-N\Vert_p \leq \varepsilon$ and all decompositions of $N$ of  translational invariant matrix product operator form. \demo
\end{enumerate}
\end{example}

\subsection{Approximate decompositions of nonnegative tensors}
\label{sec:approx-tensorranks}
In the following we define the notions of approximate $(\Omega,G)$-ranks similar to Section \ref{sec:approx-psd-ranks} using the exact $(\Omega,G)$-decompositions defined in Section \ref{sec:tensorranks}. 

Motivated by the correspondence between psd matrices and nonnegative tensors given in Equation \eqref{eq:correspondence}, we will use for $p \geq 1$ the $\ell_p$-norm, which is defined for $M \in \mathcal{K}_{n,d}$ as
$$ \Vert M\Vert_{\ell_p} = \left(\sum_{i_0, \ldots, i_n} |m_{i_0, \ldots, i_n}|^{p}\right)^{1/p}.$$
Since in the following definition of the approximate $(\Omega,G)$-decompositions of nonnegative tensors the relevant norm is the $\ell_p$-norm, we will indicate this fact by using $\ell_p$ instead of $p$.

\begin{definition} \label{def:approxnonneg}
Let $p\in [1,\infty)$, $\varepsilon > 0$ and $M \in \mathcal{K}$. We define

\setlength{\tabcolsep}{1pt}
\noindent
\begin{tabular}{c r c l}
(i) & $\textrm{rank}^{\epsilon,\ell_p}_{(\Omega,G)}(M)$ & $\coloneqq$ & $\min \{\textrm{rank}_{(\Omega,G)}(N) : \Vert M-N\Vert_{\ell_p}\leq \epsilon, N \in \mathcal{K}\}$\\[0.2cm]
(ii) & $\textrm{nn-rank}^{\epsilon,\ell_p}_{(\Omega,G)}(M)$ & $\coloneqq$ & $\min \{\textrm{nn-rank}_{(\Omega,G)}(N) : \Vert M-N\Vert_{\ell_p}\leq \epsilon, N\in \mathcal{K}\}$\\[0.2cm]
(iii) & $\textrm{psd-rank}^{\epsilon,\ell_p}_{(\Omega,G)}(M)$ & $\coloneqq$ & $\min \{\textrm{psd-rank}_{(\Omega,G)}(N) : \Vert M-N\Vert_{\ell_p}\leq \epsilon, N \in \mathcal{K}\}$\\[0.2cm]
(iv) & $\textrm{sqrt-rank}^{\epsilon,\ell_p}_{(\Omega,G)}(M)$ & $\coloneqq$ & $\min \{\textrm{sqrt-rank}_{(\Omega,G)}(N) : \Vert M-N\Vert_{\ell_p}\leq \epsilon, N \in \mathcal{K}\}$
\end{tabular}
\setlength{\tabcolsep}{6pt}
\end{definition}
Recall that the nonnegative decomposition, psd decomposition and square root decomposition only exist if $M$ is nonnegative, and thus the corresponding ranks are $\infty$ if $M$ is not nonnegative. In contrast, the approximate ranks might be finite, even if $M$ is not nonnegative.
More precisely, they are finite if and only if there exists a nonnegative tensor in the $\varepsilon$-ball of $M$ with respect to the $\ell_p$-norm.
In Example \ref{ex:tensor-ranks} we will illustrate the behavior of these ranks.

\subsection{More norms for matrices}
\label{sec:norms}

For the (non-scaled) Schatten norms, where $1\leq p\leq q\leq\infty$, we have the following inequalities \cite{Wo12b}:
\begin{equation}
\label{eq:Schatten-pq-rel}
	\Vert M\Vert_q\leq\Vert M\Vert_p \leq {\rm rank}(M)^{\frac{1}{p}-\frac{1}{q}}\Vert M\Vert_q.
\end{equation}
We denote the space of all complex hermitian $d \times d$ matrices by ${\rm Her}_d \subseteq \mathcal{M}_d$. 
For any $1\leq p\leq\infty$ we define 
$$P_p \coloneqq \{\pm \rho^{[0]}\otimes \cdots \otimes \rho^{[n]} | \textrm{ all } \rho^{[i]}\in \mathcal{M}_{d_i}^{+} \text{ and } \Vert \otimes_{i=0}^n \rho^{[i]}\Vert_p\leq 1 \} 
$$
and consider  
\be \label{eq:Bp} B_p\coloneqq\textrm{conv}(P_p)\subseteq {\rm Her}_{d_0\cdots d_n}.\ee
Note that for $p\leq q$ we have $B_p\subseteq B_q$ and already $B_1$ contains all separable states. Each $B_p$ is compact, convex, centrally symmetric and contains the origin in its interior.  We can thus understand it as the unit ball of a norm: For a set $S$ in a real vector space $V$, the {\it gauge function} $\mu_S$ is defined by 
\begin{equation}
\label{eq:gauge}
\mu_S(v)\coloneqq\inf \left\{\lambda> 0 \:| \:\frac{1}{\lambda} v\in S \right\}
\end{equation}
for $v\in V$.
If $S$ is compact, convex, centrally symmetric and has nonempty interior, the gauge function $\mu_S$ is in fact a norm (see for example Theorem 15.2 in \cite{Ro70}), and $S$ is clearly its unit ball.

We denote the gauge function of $B_p$ by $\mu_p$. We now relate these gauge functions to a multipartite version of the robustness of entanglement \cite{Vi99b}.  

\begin{definition}
Let $\rho\in\mc M_{d_0}\otimes\cdots\otimes\mc M_{d_n}$ be a state. 
The \emph{robustness of entanglement of $\rho$} is defined as 
$$
R(\rho)\coloneqq\inf\{\lambda \geq 1 \mid \rho = (1-\lambda)\rho_1+ \lambda \rho_2 , \rho_i \textrm{ separable states}\}.
$$ 
\end{definition}

Note that  this definition differs from the original robustness of entanglement by the addition of a constant 1.

\begin{proposition} 
\label{pro:gauge}
For $\rho\in {\rm Her}_{d_0}\otimes\cdots\otimes{\rm Her}_{d_n}\cong{\rm Her}_{d_0\cdots d_n}$ and $1\leq p\leq q\leq\infty$ we have:

(i) $\Vert\rho\Vert_p\leq \mu_p(\rho)$. 

(ii) $\mu_q(\rho)\leq\mu_p(\rho)\leq (d_0\cdots d_n)^{1/p-1/q} \mu_q(\rho).$

(iii) If $\rho$ is a separable state, then $\mu_p(\rho)\leq 1$. 

(iv)  If $\rho$ is a state, then $R(\rho)\leq \mu_1(\rho) \leq 2R(\rho).$

(v) Define $\mu_{\sqrt{},p}(v) = \min\{ \mu_p(\sqrt{v}) \}$. 
If $v^2=v$ then $\mu_{\sqrt{},p}(v) \leq \mu_p(v)$.

(vi) If $\rho$ is psd, then $\sqrt{\Vert \rho \Vert_{p/2}} \leq \mu_{\sqrt{},p}(\rho)$ where $\Vert \cdot \Vert_q$ is the Schatten $q$-quasinorm for values $0< q < 1$.

(vii) If $\rho$ is psd and diagonal in the standard basis, then $$\mu_{\sqrt{},p}(\rho) \leq \sqrt{\Vert \rho \Vert_{1/2}}.$$
\end{proposition}

\begin{proof}
Since the unit ball of the $p$-norm is convex, it contains $B_p$. Statement (i) follows directly from this. 
(ii) is a direct consequence of the corresponding inequalities for the $p$-norms. 
(iii) is clear from $\rho\in B_1.$ 
For the first inequality in (iv) we express $$\frac{1}{\mu_1(\rho)}\rho=\sum_i \lambda_i  \xi_i$$ as a convex combination of elements $\xi_i\in P_1$, with all $\lambda_i>0, \Vert\xi_i\Vert>0$. From the minimality of $\mu_1(\rho)$ it then follows that $\Vert\xi_i\Vert_1=1$ holds for all $i$. 
Sorting the positive and negative terms we obtain 
$$
\frac{1}{\mu_1(\rho)}\rho=r\sigma_1 - (1-r)\sigma_2
$$ 
with separable states $\sigma_i$, where $0\leq r\leq 1$ is the sum over those $\lambda_i$ with $\xi_i$ psd. Taking the trace on both sides shows $2\mu_1(\rho)r=1+\mu_1(\rho).$ Thus $\lambda\coloneqq\mu_1(\rho)r$ yields $\rho=(1-\lambda)\sigma_2+\lambda\sigma_1$ and thus $R(\rho)\leq \mu_1(\rho)$.

For the second inequality we express $\rho = (1-\lambda) \rho_1 + \lambda \rho_2$ where $\rho_i$ are separable states and $\lambda\geq 1$. Since  $\mu_1(\rho_i)\leq 1$ the second inequality follows from the triangle inequality for  $\mu_1$.\\
(v) Immediate from the definition.\\
(vi) This is also immediate, since for any square root of $\rho$ we obtain
$$\sqrt{\Vert \rho \Vert_{p/2}} =  \Vert \sqrt{\rho} \Vert_p \leq \mu_p(\sqrt{\rho})$$
where the inequality is true by (i). Minimizing over all square roots shows the statement.\\
(vii) Let $\rho$ be diagonal in the standard basis and $\sqrt{\rho}$ the unique psd square root of $\rho$. Then, defining
$K \coloneqq \Vert \rho \Vert_{1/2} = \Vert \sqrt{\rho} \Vert_1^2$ implies that $\frac{1}{\sqrt{K}} \sqrt{\rho}$ is a convex combination of rank one projectors onto the standard basis. Hence, $\sqrt{\rho} \in \sqrt{K} B_1 \subseteq \sqrt{K} B_p$ which shows the inequality.
\end{proof}

\begin{remark}
Using the relation between $\mu_1$ and the robustness of entanglement (see Proposition \ref{pro:gauge} (iv)), one can show that there does not exist a dimension independent upper bound for any Schatten $p$-(quasi)norm ($0 < p < \infty$) considering arbitrary psd matrices.

To see this, consider the bipartite matrix tensor product space $\mathcal{M}_d \otimes \mathcal{M}_d$ and set $\rho = v v^{*}$ with
$$ v = \frac{1}{\sqrt{d}} \sum_{i=1}^{d} v_i \otimes v_i \in \mathbb{C}^d \otimes \mathbb{C}^d$$
where $\{v_1, \ldots, v_d\}$ is an orthogonal basis of $\mathbb{C}^d$.
On the one hand, $\Vert \rho \Vert_p = 1$, but on the other hand, by [22, Eq. (30)],
$$\mu_1(\rho) \geq R(\rho) = \left(\sum_{i=1}^{d} \frac{1}{\sqrt{d}}\right)^2 = d.$$
To the best of our knowledge, there is no dimension-independent upper bound of $\mu_{\sqrt{},p}$ or $\mu_{p}$ with respect to the Schatten $p$-norm for arbitrary psd matrices and $p > 1$.
\demo
\end{remark}

\subsection{Upper bounds for approximate ranks}
\label{ssec:general}
We now prove an upper bound for the approximate $(\Omega,G)$-rank which  only depends on the gauge function value of the matrix, and the approximation error $\varepsilon$. 
Recall that $\Omega$ is a connected wsc and $G$ a free group action on $\Omega$.

\begin{theorem}\label{thm:weirdconv}
Let $1 < p \leq 4/3$, $p=2$ or $4 \leq p < \infty$ and $\varepsilon > 0$. Assume $M \in \mathrm{Her}_{d_0 \cdots d_n}$ is $G$-invariant. Then
\begin{enumerate}[label=(\alph*)]
	\item $\displaystyle \rank^{\epsilon,p}_{(\Omega,G)}(M) \leq \left \lceil C_p \cdot \left(\frac{2 \mu_{p}(M)}{\epsilon}\right)^{\frac{p}{p-1}}\right\rceil \cdot |G| \quad \text{if }1 < p\leq 4/3$
	\item $\displaystyle \rank^{\epsilon,p}_{(\Omega,G)}(M) \leq \left \lceil D_p \cdot \left(\frac{2 \mu_{p}(M)}{\epsilon}\right)^2\right\rceil \cdot |G| \quad \quad \text{if } p=2 \text{ or } 4 \leq p < \infty$
\end{enumerate}
where $C_p, D_p$ are constants defined in Theorem \ref{thm:approxcara}.
\end{theorem}
 
The proof is a straightforward application of Theorem \ref{thm:approxcara} with set $P_p$ to construct an $\Omega$-decomposition. To obtain a $G$-invariant decomposition, we apply the construction used in \cite[Thm. 13]{De19d}.

\begin{proof}
First let $p = 2$ or $4 \leq p < \infty$.
We prove the statement for the case that $\mu_{p}(M)=1$; the general case follows by replacing 
$M$ with $M/\mu_p(M)$. 
First we consider the case of trivial group action. 
By assumption $M\in \textrm{conv}(P_p)$, and $\textrm{diam}(P_p)=2$. 
By Theorem \ref{thm:approxcara} we find an $M'$ such that $\Vert M- M'\Vert\leq \epsilon$ and $M'$  is a convex combination of at most $\lceil 4 D_p/\epsilon^2\rceil$ elements of $P_p$. This implies that 
$$
\rank_{\Omega}(M')\leq \rank_{\Sigma_n}(M')\leq \lceil 4 D_p/\epsilon^2\rceil, 
$$
where the first inequality follows from \cite[Prop. 36]{De19d}.
We now involve a group action and apply the construction of \cite[Thm. 13]{De19d} to $M'$. This gives rise to an $(\Omega,G)$-decomposition of an element $M'' =  \frac{1}{|G|}\sum_{g} g M'$ with $$\textrm{rank}_{(\Omega,G)}(M'') \leq \lceil 4 D_p/\epsilon^2\rceil |G|.$$
It remains to prove that $M''$ is contained in the $\epsilon$-neighborhood of $M$: 
\be
\nn \Vert M''-M\Vert_p &=& \left\Vert \frac{1}{|G|} \sum_g (g M'- g M)\right\Vert_p \\
\nn &\leq& \frac{1}{|G|}  \sum_g \Vert g M' - g M\Vert_p 
\leq \epsilon
\ee
where we used the invariance $M = \frac{1}{|G|}\sum_g g M$, 
that the norm is unitarily invariant $\Vert g M' - g M\Vert_p = \Vert M' -  M\Vert_p$, 
and that $M'$ is in the $\epsilon$-ball around $M$. The case $1 < p \leq 4/3$ is analogous.

\end{proof}

\begin{remark} \label{rem:repeat}
Note that Theorem \ref{thm:weirdconv} and all following theorems state upper bounds for convex combinations of the type
$$ M_k = \frac{1}{k} \sum_{i = 1}^{k} X_i,$$
where $X_i \in P_p$. The sequence $\{X_i\}_{i=1}^{k}$ might contain a repetition of elements which would decrease the rank of the decomposition even more, but our estimates do not exploit this fact.
\demo
\end{remark}

\subsection{Upper bounds for approximate ranks of psd matrices}
\label{ssec:psd}

We now provide  upper bounds for the purification rank and the quantum square root rank of psd matrices. 
More specifically, we provide an upper bound for the quantum square root rank, which is itself an upper bound of the purification rank (see Remark \ref{rem:purisqrtrk} (i)).
Recall that $\Omega$ is a connected wsc and $G$ a free group action on $\Omega$.
\begin{corollary}
\label{cor:psdrkbound}
Let $1 < p \leq 4/3$, $p=2$ or $4 \leq p < \infty$ and $\varepsilon > 0$. Let $\rho \in \mathcal{M}^{+}_{d_0 \cdots d_n}$ be $G$-invariant. 
Then 
\renewcommand{\arraystretch}{1.2}
\begin{enumerate}[label=(\alph*), wide, labelwidth=!, labelindent=0pt]
	\item $\textrm{puri-rank}_{(\Omega,G)}^{\epsilon,p} (\rho) \leq \left\lceil C_p \cdot \left(\frac{2 }{\sqrt{1+\epsilon / \mu^2_{\sqrt{},p}(\rho)} - 1}\right)^{\frac{p}{p-1}}\right\rceil \cdot |G| \: \text{ if } 1 < p \leq 4/3$
	\item $\textrm{puri-rank}_{(\Omega,G)}^{\epsilon,p} (\rho) \leq \left\lceil D_p \cdot \left(\frac{2 }{\sqrt{1+\epsilon / \mu^2_{\sqrt{},p}(\rho)} - 1}\right)^2\right\rceil \cdot |G| \quad \text{ if } \begin{array}{l}p = 2 \text{ or } \\ 4 \leq p < \infty\\ \end{array}$
\end{enumerate}
and the same upper bounds hold for $\textrm{q-sqrt-rank}_{(\Omega,G)}^{\varepsilon, p}$.
\renewcommand{\arraystretch}{1.7}
\end{corollary}

\begin{proof}
We apply Theorem \ref{thm:weirdconv} to the square root of $\rho$ (called $M$) which realises  $\mu_{\sqrt{},p}(\rho)$. We have that $M/\mu_{\sqrt{},p}(\rho) \in \conv(P_p)$ and thus $\Vert M\Vert_p\leq \mu_{\sqrt{},p}(\rho)$ (defined in Proposition \ref{pro:gauge} (v)). 
We now choose an element $M'$ which is $\mu_{\sqrt{},p}(\rho) \cdot \delta$ - close 
to $M$ with $\delta=\sqrt{1+\epsilon/\mu^2_{\sqrt{},p}(\rho)}-1$.
We define $\rho'= M'^2$ and compute
\be
\Vert \rho -\rho'\Vert_p &=& \Vert M^2-M'^2\Vert_p\nn  \\
\nn &\leq &  \Vert M-M'\Vert_p \cdot (2 \Vert M\Vert_p+ \Vert  M-M'\Vert_p )\\
\nn &\leq & \mu_{\sqrt{},p}(\rho)^2 \cdot \delta \cdot (2 +  \delta  ) \leq \epsilon.
\ee 
Using that $$\textrm{puri-rank}_{(\Omega,G)} (\rho') \leq \textrm{q-sqrt-rank}_{(\Omega,G)} (\rho') \leq \textrm{rank}_{(\Omega,G)} (M')$$ and the fact that $\textrm{rank}_{(\Omega,G)} (M')$ is upper bounded by Theorem \ref{thm:weirdconv} we obtain the result. 
\end{proof}

Note that this result in fact upper bounds the quantum square root rank, which may be arbitrarily larger than the purification rank \cite{De19}. 

\subsection{Upper bounds for approximate ranks of separable states}
\label{ssec:sep}

An upper bound for separable states can be obtained without the use of a gauge function. This is possible because we can directly conclude that a separable state is in the convex hull of $$\left\{+\rho^{[0]} \otimes \cdots \otimes \rho^{[n]}: \rho^{[i]} \in \mathcal{M}_{d_i}^{+} \text{ and } \Vert\otimes_{i=0}^{n} \rho^{[i]}\Vert_p \leq 1 \right\} \subseteq P_p.$$
Using that $\textrm{sep-rank}_{(\Omega, G)}$ upper bounds $\textrm{rank}_{(\Omega,G)}$ and $\textrm{puri-rank}_{(\Omega,G)}$, we can conclude that this upper bound holds for the other two ranks in the approximate case too.
Recall that $\Omega$ is a connected wsc and $G$ a free group action on $\Omega$.

\begin{proposition}
\label{prop:seprkbound}
Let $1 < p \leq 4/3$, $p=2$ or $4 \leq p < \infty$ and $\varepsilon > 0$. Further let ${\rho \in \mathrm{SEP}_{d_0, \ldots, d_n}}$ be a $G$-invariant separable state.
Then 
\begin{enumerate}[label=(\alph*)]
	\item $\displaystyle\textrm{sep-rank}^{\epsilon,p}_{(\Omega,G)}(\rho) \leq \lceil C_p \cdot (2/\epsilon)^{\frac{p}{p-1}}\rceil \cdot |G| \quad \: \text{if } 1 < p \leq 4/3$
	\vspace{0.2cm}
	\item $\displaystyle\textrm{sep-rank}^{\epsilon,p}_{(\Omega,G)}(\rho) \leq \lceil D_p \cdot (2/\epsilon)^2\rceil \cdot |G| \quad \quad \text{if } p=2 \text{ or } 4 \leq p < \infty$
\end{enumerate}
The same upper bound holds for $\textrm{rank}^{\epsilon,p}_{(\Omega,G)}(\rho)$ and $\textrm{puri-rank}^{\epsilon,p}_{(\Omega,G)}(\rho)$, too. 
\end{proposition}

\begin{proof}
It is proven exactly as Theorem \ref{thm:weirdconv} by using that separable states are a convex combination of product states, which are a subset of $P_p$. 
The rest follows from \cite[Thm.\ 43]{De19d} and \cite[Cor.\ 44]{De19d}. 
\end{proof}

Note again that the separable rank may be arbitrarily larger than the rank and the purification rank \cite{De19}.

\section{Disappearance of separations in the approximate case}
\label{sec:disap}

We now turn to study relations between the different notions of ranks, especially their (lack of) separations in the approximate case. 

Let $\mathcal X$ be an arbitrary set and $f,g\colon\mathcal{X} \to\mathbb N$ two functions. We say that there is a \textit{separation} between $f$ and $g$, and write $f \ll g$, if there exists a sequence  in $  \mathcal{X}$ along which $f$ is bounded but $g$ is not. This implies that the values of $g$ cannot be upper bounded by a function that only depends on $f$. 
In the exact case, there are several examples of separations between different ranks --- for example on the set  $$ 
\mathcal{X} \coloneqq \bigcup_{d \in \mathbb{N}} \mathcal{M}_{d^{n+1}}^+
$$
one has $\textrm{rank}_{\Lambda_n} \ll \textrm{puri-rank}_{\Lambda_n}$ \cite{De13c,De19}. In other words, there is a separation between the operator Schmidt rank and the purification rank.
As another example, on the set
$$\mathcal{Y} \coloneqq \{M \in \mathbb{C}^d \otimes \mathbb{C}^d: M \textrm{ nonnegative and } d \in \mathbb{N}\}
$$
one has $\textrm{rank}_{\Lambda_1} \ll \textrm{psd-rank}_{\Lambda_1}$ \cite{Fa14,Go12} which comes from the fact that $\Lambda_1$-decompositions on $\mathbb{C}^d \otimes \mathbb{C}^d$ correspond to different notions of matrix factorizations on $\mathcal{M}_d$. Other examples of separations are collected in Table \ref{tab:RelationsMatrixFact}.

In the following we will show that many separations of ranks of psd matrices 
disappear (Section \ref{sec:psd-disappear}), and the same happens for nonnegative tensors (Section \ref{sec:nonneg-dissappear}), including both examples above. This is an immediate consequence of the fact 
the approximate ranks admit an upper bound which is independent of the dimension of the ambient space.

As in previous section, we  fix a connected wsc $\Omega$ and a free group action $G$ acting on $\Omega$.

\subsection{Positive semidefinite matrices}
\label{sec:psd-disappear}

In the following we show that the separation between several ranks of psd matrices vanishes in the approximate case. 
The strategy is simple: we will use Theorem \ref{thm:weirdconv}, Proposition \ref{prop:seprkbound} and Corollary \ref{cor:psdrkbound} to upper bound the ranks independently of the matrix dimension. Since all ranks are bounded functions, it follows that many separations vanish.

For simplicity we assume that $d_i = d$ and hence  consider the space $$\mathcal{M}_d \otimes \cdots \otimes \mathcal{M}_d \cong \mathcal{M}_{d^{n+1}}.$$
 The result can be extended in a straightforward manner to the case that the parts of the tensor product have different dimensions.

\begin{corollary}
\label{cor:sepMat}
Let $G$ be a free group action on $\Omega$. Let $\varepsilon > 0$, $K \in \mathbb{N}$ and $1 < p \leq 4/3$, $p = 2$ or $4 \leq p < \infty$.
We define the set
$$
\mathcal{X}_{K} \coloneqq \{\rho \in \mathcal{M}_{d^{n+1}}^{+}: d \in \mathbb{N} \text{ and } \mu_{\sqrt{},p}(\rho) \leq K\}.
$$
Then the following holds on $\mathcal{X}_K$:
\begin{center}
\begin{tabular}{c l c l}
	(i) & $\textrm{rank}_{(\Omega,G)}^{\varepsilon, p}$ & $\notll$ & $\textrm{puri-rank}_{(\Omega,G)}^{\varepsilon, p}$\\
	(ii) & $\textrm{puri-rank}_{(\Omega,G)}^{\varepsilon, p}$ & $\notll$ & $\textrm{q-sqrt-rank}_{(\Omega,G)}^{\varepsilon, p}$
\end{tabular}
\end{center}
\vspace{0.5cm}
Further, we define the set 
$$
\mathcal{X}_\textrm{sep} \coloneqq \bigcup_{d \in \mathbb{N}} \mathrm{SEP}_{n,d}. 
$$
The following holds on $\mathcal{X}_{sep}$: 
\begin{center}
\setlength{\tabcolsep}{5pt}
\begin{tabular}{c l c l}
	(iii) & $\textrm{rank}_{(\Omega,G)}^{\varepsilon, p}$ & $\notll$ & $\textrm{sep-rank}_{(\Omega,G)}^{\varepsilon, p}$\\
	(iv) & $\textrm{puri-rank}_{(\Omega,G)}^{\varepsilon, p}$ & $\notll$ & $\textrm{sep-rank}_{(\Omega,G)}^{\varepsilon, p}$
\end{tabular}
\end{center}
\end{corollary}

Note that (i) and (ii) imply that $\textrm{rank}_{(\Omega,G)}^{\varepsilon, p} \notll \textrm{q-sqrt-rank}_{(\Omega,G)}^{\varepsilon, p}$ on $\mathcal{X}_{K}$.

\begin{proof}
(i)-(ii) For $\rho \in \mathcal{M}_{d^{n+1}}^{+}$ we have $\textrm{rank}_{(\Omega,G)}(\rho) \leq \textrm{puri-rank}_{(\Omega,G)}(\rho)^2$ \cite[Prop. 29]{De19d}.
By the basic properties of $\textrm{q-sqrt-rank}_{(\Omega,G)}$ (Remark \ref{rem:purisqrtrk} (i)) we obtain
$$ 
\sqrt{\textrm{rank}_{(\Omega,G)}^{\varepsilon, p}(\rho)} \leq \textrm{puri-rank}_{(\Omega,G)}^{\varepsilon, p}(\rho) \leq \textrm{q-sqrt-rank}_{(\Omega,G)}^{\varepsilon,p}(\rho).
$$
Since Corollary \ref{cor:psdrkbound} upper bounds $\textrm{q-sqrt-rank}_{(\Omega,G)}^{\varepsilon,p}$ by a constant
which is independent of the dimension of $\rho \in \mathcal{X}_K$, all  ranks mentioned above are upper bounded on $\mathcal{X}_K$.

(iii)-(iv) Let $\rho \in \mathcal{X}_\text{sep}$. Using again \cite[Prop. 29]{De19d} we have that
\be \nn \textrm{rank}_{(\Omega,G)}^{\varepsilon, p}(\rho) &\leq& \textrm{sep-rank}_{(\Omega,G)}^{\varepsilon, p}(\rho), \\ \nn \textrm{puri-rank}_{(\Omega,G)}^{\varepsilon, p}(\rho) &\leq& \textrm{sep-rank}_{(\Omega,G)}^{\varepsilon, p}(\rho). \ee
By Proposition \ref{prop:seprkbound} $\textrm{sep-rank}_{(\Omega,G)}$ is upper bounded by a constant which is again independent of the dimension of $\rho$. 
\end{proof}

Note that the previous result is not making any statement about the Schatten 1-norm, since the approximate Carath\'eodory Theorem is not applicable for this norm. Using the results from Theorem \ref{thm:weirdconv}, Proposition \ref{prop:seprkbound} and Equation \eqref{eq:Schatten-pq-rel}, we now give an upper bound of the approximate rank and approximate separable rank in Schatten 1-norm, which is however dimension dependent.

\begin{corollary}
\label{cor:Schatten1}
Let $M \in \mathrm{Her}_{d^{n+1}}$ and $\rho \in {\rm SEP}_{n,d}$ both be $G$-invariant.
Then: 
\begin{enumerate}[label=(\roman*)]
	\item $\textrm{rank}_{(\Omega,G)}^{\varepsilon, 1}(M) \leq {d^{n+1}} \cdot \textrm{rank}_{(\Omega,G)}^{\varepsilon, 2}(M) \leq \left \lceil D_2 \cdot \left(\frac{2 \mu_{2}(M)}{\varepsilon}\right)^2\right\rceil \cdot |G| \cdot d^{n+1}$
	\item $\textrm{sep-rank}_{(\Omega,G)}^{\varepsilon, 1}(\rho) \leq {d^{n+1}} \cdot \textrm{sep-rank}_{(\Omega,G)}^{\varepsilon, 2}(\rho) \leq \lceil D_2 \cdot (2/\epsilon)^2\rceil \cdot |G| \cdot d^{n+1}$
\end{enumerate}
\end{corollary}

\begin{proof}
Let $A, A_k \in \mathcal{M}_{d^{n+1}}, \mathcal S\subseteq \mathcal{M}_{d^{n+1}}$ be chosen as in Theorem \ref{Thm:ApproxCaratheodory}. Using Equation \eqref{eq:Schatten-pq-rel} and  optimizing over all valid $p$-values gives the bound
$${\Vert A-A_k\Vert_1 \leq \sqrt{d^{n+1}} \cdot \Vert A-A_k\Vert_2},$$ or equivalently
$$ \Vert A-A_k\Vert_1 \leq \varepsilon \quad \text{ if } \quad k \geq d^{n+1} \cdot \left\lceil D_2 \cdot \left(\frac{\textrm{diam}(\mathcal S)}{\epsilon}\right)^{2}\right\rceil.$$
Applying this statement to the proofs of Theorem \ref{thm:weirdconv} and Proposition \ref{prop:seprkbound} shows the statement. 
\end{proof}

\begin{example} We now apply Corollary \ref{cor:Schatten1} to the running examples of Example \ref{ex:decompositions1} and Example \ref{ex:approxdecompositions}.
\begin{enumerate}[label=(\roman*), wide=\parindent, labelwidth=!]
	\item Consider the line $\Omega = \Lambda_n$ of size $n$ and $G$ the trivial group. For the approximate operator Schmidt rank in Schatten $1$-norm of $\rho \in \mathcal{M}_{d^{n+1}}^{+}$ it holds that 
	$$
	\textrm{rank}_{\Lambda_n}^{\varepsilon, 1}(\rho) \leq \left\lceil D_2 \cdot \left(\frac{2 \mu_{2}(\rho)}{\varepsilon}\right)^2\right\rceil \cdot d^{n+1} .
	$$
	If $\rho \in \textrm{SEP}_{n,d}$, then
	$$
	\textrm{sep-rank}_{\Lambda_n}^{\varepsilon, 1}(\rho) \leq \lceil D_2 \cdot (2/\epsilon)^2\rceil \cdot d^{n+1}.
	$$
	Note that, in the exact case, $\textrm{rank}_{\Lambda_n}$ is also bounded by a constant times $d^{n+1}$ \cite[Prop. 49]{De19}. In contrast, the upper bound of  $\textrm{sep-rank}_{\Lambda_n}$ in the exact case is linear in $d^{2(n+1)}$. 
Hence, for fixed $\varepsilon$ and sufficiently large dimension $d$ this yields a better upper bound in the approximate case.
	\item Consider the circle of $n$ elements, $\Omega = \Theta_n$, together with the cyclic group $C_n$. Further let $\rho \in \mathcal{M}_{d^n}^{+}$ be $C_n$-invariant. Since $|C_n| = n$ we obtain
	$$\textrm{rank}_{(\Theta_n,C_n)}^{\varepsilon, 1}(\rho) \leq \left\lceil D_2 \cdot \left(\frac{2 \mu_{2}(\rho)}{\varepsilon}\right)^2\right\rceil \cdot n \cdot d^{n}$$
	and
	$$\textrm{sep-rank}_{(\Theta_n,C_n)}^{\varepsilon, 1}(\rho) \leq \lceil D_2 \cdot (2/\epsilon)^2\rceil \cdot n \cdot d^{n}.$$ \demo
\end{enumerate}
\end{example}

\subsection{Nonnegative tensors}
\label{sec:nonneg-dissappear}
In the following we use the correspondence between nonnegative tensors and diagonal psd matrices given in Equation \eqref{eq:correspondence} to show that the separation between $\textrm{nn-rank}_{(\Omega,G)}$, $\textrm{rank}_{(\Omega,G)}$ and $\textrm{psd-rank}_{(\Omega,G)}$ of nonnegative tensors vanishes in the approximate case (Corollary \ref{cor:sepTen}).

As in Section \ref{sec:tensorranks}, we consider the space $\mathcal{K}_{n,d} \coloneqq \bigotimes_{i=0}^{n} \mathbb{C}^{d}$ equipped with the $\ell_p$-norm. In contrast to psd matrices, which are equipped with the Schatten $p$-norm, the disappearance of separations for nonnegative tensors equipped with the $\ell_p$-norm can be proven for all $1 < p < \infty$.
Recall that $\Omega$ is a connected wsc and $G$ a free group action.

We start with a preparatory lemma, which is an immediate consequence of \cite[Prop. 29]{De19d} and the relations  given in Remark \ref{rem:rank-relations}.

\begin{lemma}
\label{prop:tens-rank-ineq}
Let $M \in \mathcal{K}_{n,d}$ and $p \in [1, \infty)$.
Then the following holds:\\
\setlength{\tabcolsep}{5pt}
\begin{tabular}{c l c l}
(i) & $\textrm{rank}_{(\Omega, G)}^{\varepsilon, \ell_p}(M)$ & $\leq$ & $\textrm{nn-rank}_{(\Omega, G)}^{\varepsilon, \ell_p}(M)$\\
(ii) & $\textrm{psd-rank}_{(\Omega, G)}^{\varepsilon, \ell_p}(M)$ & $\leq$ & $\textrm{nn-rank}_{(\Omega, G)}^{\varepsilon, \ell_p}(M)$\\
(iii) & $\textrm{rank}_{(\Omega, G)}^{\varepsilon, \ell_p}(M)$  & $\leq$ & $\textrm{psd-rank}_{(\Omega, G)}^{\varepsilon, \ell_p}(M)^2$\\
(iv) & $\textrm{psd-rank}_{(\Omega,G)}^{\varepsilon,\ell_p}(M)$ & $\leq$ & $\textrm{sqrt-rank}_{(\Omega,G)}^{\varepsilon,\ell_p}(M)$
\end{tabular}
\end{lemma}

Note that it is not clear if a similar result to Remark \ref{rem:rank-relations} holds for  approximate ranks. This is due to the fact that the $\varepsilon$-ball of diagonal matrices also contains non-diagonal matrices with a possibly smaller rank than the diagonal ones.

\begin{corollary}
\label{cor:sepTen}
Let $G$ be a free group action on $\Omega$.  Let $\varepsilon > 0$, $K \in \mathbb{N}$ and $1 < p < \infty$.
We define the set
$$
\mathcal{Y}_K \coloneqq \{M \in \mathcal{K}_{n,d} : d \in \mathbb{N}, M \text{ nonnegative and } \Vert M\Vert_{\ell_1} \leq K\}.
$$
Then the following holds on $\mathcal{Y}_K$: 
\begin{center}
\setlength{\tabcolsep}{5pt}
\begin{tabular}{c l c l}
	(i) & $\textrm{rank}_{(\Omega,G)}^{\varepsilon, \ell_p}$ & $\notll$ & $\textrm{psd-rank}_{(\Omega,G)}^{\varepsilon, \ell_p}$\\
	(ii) & $\textrm{psd-rank}_{(\Omega,G)}^{\varepsilon, \ell_p}$ & $\notll$ & $\textrm{nn-rank}_{(\Omega,G)}^{\varepsilon, \ell_p}$
\end{tabular}
\end{center}
\vspace{0.5cm}
Further, we define the set
$$\mathcal{Y}_{\sqrt{},K} \coloneqq \{M \in \mathcal{K}_{n,d}: d \in \mathbb{N}, M \text { nonnegative and } \mu_{\sqrt{},p}(\sigma) \leq K\}$$
where $\sigma$ denotes the corresponding diagonal matrix of $M$. 
Then the following holds on $\mathcal{Y}_{\sqrt{},K}$:\\
\begin{center}
\setlength{\tabcolsep}{5pt}
\begin{tabular}{c l c l}
	(iii) & $\textrm{psd-rank}_{(\Omega,G)}^{\varepsilon, \ell_p}$ & $\notll$ & $\textrm{sqrt-rank}_{(\Omega,G)}^{\varepsilon, \ell_p}$\\
	(iv) & $\textrm{rank}_{(\Omega,G)}^{\varepsilon, \ell_p}$ & $\notll$ & $\textrm{sqrt-rank}_{(\Omega,G)}^{\varepsilon, \ell_p}$
\end{tabular}
\end{center}
\end{corollary}
Note that (i) and (ii) imply that $\textrm{rank}_{(\Omega,G)}^{\varepsilon,\ell_p} \notll \textrm{nn-rank}_{(\Omega,G)}^{\varepsilon, \ell_p}$ on $\mathcal{Y}_{K}$.

\begin{proof}
(i)-(ii) Let $M \in \mathcal{Y}_{K}$. By Proposition \ref{prop:tens-rank-ineq}, we have that $\textrm{rank}_{(\Omega,G)}^{\varepsilon, \ell_p}(M)$ and $\textrm{psd-rank}_{(\Omega,G)}^{\varepsilon,\ell_p}(M)$ are upper bounded by  $\textrm{nn-rank}_{(\Omega,G)}^{\varepsilon, \ell_p}(M)$. To obtain an upper bound of $\textrm{nn-rank}_{(\Omega,G)}^{\varepsilon, \ell_p}(M)$ we restrict to $p=2$, as the other cases are analogous.
We set $M' = M/{\Vert M\Vert_{\ell_1}}$ and $\varepsilon' = \varepsilon/K$. Then the corresponding diagonal matrix $\sigma'$ is separable with $\Vert \sigma' \Vert_1 = 1$ (it can be written for example as a convex combination of rank one projectors to the standard basis).
Note that any $\varepsilon'$-approximation of $M'$ immediately leads to an $\varepsilon$-approximation of $M$. 
Using 
$$
P_2' \coloneqq \left\{\rho_0 \otimes \cdots \otimes \rho_n : \rho_i \in \mathcal{M}_d^{+} \text{ diagonal and } \Vert\bigotimes_{i=0}^{n} \rho_i\Vert_2 \leq 1\right\} \subseteq P_2
$$
instead of $P_2$ in the proof of Theorem \ref{thm:weirdconv}, we obtain that 
$$
\textrm{nn-rank}_{(\Omega,G)}^{\varepsilon, \ell_2}(M) \leq \lceil D_2 \cdot (2K/\epsilon)^2\rceil \cdot |G|
$$
which is independent of the ambient dimension.

Since the Schatten $p$-norm on the space of diagonal matrices is equivalent to the $\ell_p$-norm, we can apply Theorem \ref{thm:approxcara} and Theorem \ref{thm:weirdconv} for $1 < p < \infty$ (see Remark \ref{rem:lp-approxCara} for details).

(iii)-(iv) Let $M \in \mathcal{Y}_{\sqrt{}, K}$. Again, by Proposition \ref{prop:tens-rank-ineq}, $\textrm{psd-rank}_{(\Omega,G)}^{\varepsilon, p}(M)$ is upper bounded by $\textrm{sqrt-rank}_{(\Omega,G)}^{\varepsilon,p}(M)$. 
Applying Corollary \ref{cor:psdrkbound} to the corresponding diagonal matrix $\sigma$ gives an upper bound for $\textrm{q-sqrt-rank}_{(\Omega,G)}^{\varepsilon,p}(\sigma)$. This is again independent of the ambient dimension.
\end{proof}

\begin{remark}
(i) Both sets $\mathcal{X}_K$ and $\mathcal{Y}_{\sqrt{}, K}$ are not convex in general. In particular, by Proposition \ref{pro:gauge} (v), (vi) 
$$\left\{M \in \mathcal{K}_{n,d} : d \in \mathbb{N}, M \text{ nonneg. and } \Vert M\Vert_{\ell_{1/2}} \leq K^2 \right\} \subseteq \mathcal{Y}_{\sqrt{},K}$$
and
$$\mathcal{Y}_{\sqrt{},K} \subseteq \left\{M \in \mathcal{K}_{n,d} : d \in \mathbb{N}, M \text{ nonneg. and } \Vert M\Vert_{\ell_{p/2}} \leq K^2 \right\}$$
where $\Vert \cdot \Vert_{\ell_q}$ is the entrywise $q$-quasinorm for $0 < q < 1$. This shows that $\mathcal{Y}_{\sqrt{},K}$ coincides with the (non-convex) $\ell_{1/2}$-ball of radius $K^2$ in using the normalization $\mu_{\sqrt{}, 1}$.

(ii) Similarly to  psd matrices  (Corollary \ref{cor:Schatten1}) we can upper bound  the approximate ranks for nonnegative tensors in the $\ell_1$-norm.  The correspondence between diagonal psd matrices and nonnegative matrices yields for a nonnegative and $G$-invariant  $M \in \mathcal{K}_{n,d}$ the inequality 
$$\textrm{rank}_{(\Omega,G)}^{\varepsilon, \ell_1}(M) \leq d^{n+1} \cdot \textrm{rank}_{(\Omega,G)}^{\varepsilon, \ell_2}(M). 
$$
This is the same $d$-dependence as the one of the exact Carath\'eodory Theorem \cite{We94}.
Similar results appear for all other ranks of nonnegative tensors.
\demo
\end{remark}

\begin{example}
\label{ex:tensor-ranks} 
We now give some concrete examples of decompositions whose separations between ranks disappear in the approximate case. From now on we consider the space $\mathcal{K}_{1,d} = \mathbb{C}^d \otimes \mathbb{C}^d \cong \mathcal{M}_{d}$ and $\Lambda_1$-decompositions, i.e.\ decompositions of the form 
$$M = \sum_{\alpha=1}^{r} v^{[0]}_\alpha \otimes v^{[1]}_\alpha,$$
where $v^{[i]}_\alpha \in \mathbb{C}^{d}$. 

Recall that for $M$ nonnegative, $\textrm{nn-rank}_{\Lambda_1}(M)$ is the smallest integer $r$ such that there exists a decomposition of $M$ with $v^{[i]}_\alpha \in \mathbb{R}^d_{+}$ for all ${\alpha = 1, \ldots, r}$ (i.e.\ vectors with nonnegative entries). Furthermore, $\textrm{psd-rank}_{\Lambda_1}(M)$ is the smallest integer $r$ such that there exist $A^{[i]}_{j} \in \mathcal{M}_r^{+}$ for $j = 1, \ldots, d$ and $i = 0,1$ which generate a decomposition
$$ M_{kl} = \sum_{\alpha_0,\alpha_1 = 1}^{r} \left(A^{[0]}_{k}\right)_{\alpha_0, \alpha_1} \cdot \left(A^{[1]}_{l}\right)_{\alpha_0,\alpha_1} = \tr\left( \left(A^{[0]}_k\right) \cdot \left(A^{[1]}_l\right)^t \right) $$
Assume $1 < p \leq 4/3$, $p =2$ or $4 \leq p < \infty$ and $\varepsilon > 0$ fixed. The separations are based on examples mentioned in \cite{Fa14}.

\begin{enumerate}[label=(\roman*), wide=\parindent, labelwidth=!]
	\item Consider the normalized Euclidean distance matrix $${U_d = M_d/\Vert M_d\Vert_{\ell_1} \in \mathcal{M}_d},$$ where $M_d$ is defined as
	$$ 
	M_d \coloneqq ((i-j)^2)_{i,j=1}^{d}.
	$$
	It was shown that $\textrm{sqrt-rank}_{\Lambda_1}(M_d) = \textrm{psd-rank}_{\Lambda_1}(M_d) = 2$ \cite{Fa14} and
	$$\textrm{nn-rank}_{\Lambda_1}(M_d) \geq \log_2(d).$$ Obviously the same statement is true for $U_d$. Hence, we obtain $${\textrm{sqrt-rank}_{\Lambda_1} \ll \textrm{nn-rank}_{\Lambda_1}} \quad \text{ and } \quad \textrm{psd-rank}_{\Lambda_1} \ll \textrm{nn-rank}_{\Lambda_1}.$$
	Since $\Vert U_d\Vert_{\ell_1} = 1$, Corollary \ref{cor:sepTen} shows that the separations for this example vanish, since $\textrm{nn-rank}_{\Lambda_1}^{\varepsilon, \ell_p}(U_d)$ can be upper bounded independently of $d$.
	\item Let $S_d$ be the slack matrix of a $d$-gon in the plane. We define the normalized slack matrix as $V_d \coloneqq S_d/{\Vert S_d\Vert_{\ell_1}}$.
	It was shown that ${\textrm{rank}_{\Lambda_1}(S_d) = 3}$ and $\textrm{psd-rank}_{\Lambda_1}(S_d)$ diverges if $d$ goes to infinity \cite{Go12}. Obviously, the same holds for $V_d$. Hence $\textrm{rank}_{\Lambda_1} \ll \textrm{psd-rank}_{\Lambda_1}$.
	
	Since $\Vert V_d\Vert_{\ell_1} = 1$, Corollary \ref{cor:sepTen} shows that the separation for this example vanishes, i.e.\ $\textrm{psd-rank}_{\Lambda_1}^{\varepsilon, \ell_p}(V_d)$ can be upper bounded independently of $d$. \demo
\end{enumerate}
\end{example}

Note that in these examples the normalization is important, because both $\Vert S_d \Vert_{\ell_1}$ and $\Vert M_d \Vert_{\ell_1}$ are unbounded as a function of $d$.  
Further note that this discussion can  be extended to $\Lambda_n$-decompositions with $n \geq 1$. One application fulfilling the normalization condition is the interpretation of a tensor $M \in \mathcal{K}_{n,d}$ as a probability mass function $P(X_0, \ldots, X_n)$ over $n+1$ discrete random variables, taking values in $\{1, \ldots, d\}$ and setting
$$m_{i_0, \ldots, i_n} \coloneqq P(X_0 = i_0, \ldots, X_n = i_n).$$
In this case $M$ is nonnegative and bounded with $\Vert M \Vert_{\ell_1} = 1$. The nonnegative $\Lambda_n$-decomposition corresponds to a \emph{hidden Markov model} \cite{Gl19}.

\section{Algorithm to find the approximate decomposition}
\label{sec:algo}

The sequence approximating a matrix in the  convex hull in Theorem \ref{Thm:ApproxCaratheodory} can be computed by a deterministic algorithm presented in \cite{Iv19} for uniformly smooth Banach spaces. In the following we give an explicit description of this algorithm for Schatten $p$-classes.

The Schatten $p$-norm is everywhere differentiable except for $0$. 
Let $X,Y \in \mathcal{M}_d \setminus \{0\}$ be two arbitrary matrices. Further, let
$$ X = U \cdot \Sigma \cdot V^{*}$$
be a singular value decomposition, in particular, $U,V \in \mathcal{M}_{d, r}$ are isometries where $r \leq d$ and 
$$\Sigma = \diag(\lambda_1, \ldots, \lambda_r) \in \mathcal{M}_r$$
is a diagonal matrix with positive entries.
We  also denote the $i$th column of $U$ and $V$ by $u_i$ and $v_i$, respectively. Then, the directional derivative of $\Vert \cdot \Vert_p$ at $X$ in direction $Y$ can be evaluated (see Ref. \cite{Wa92} for details) by
\be \label{eq:partialderivative} \mathcal{D}_Y \Vert \cdot \Vert_p \:\biggr\rvert_{X} = \frac{1}{\Vert X \Vert_p^{p-1}} \cdot \sum_{i=1}^{r} \lambda_i^{p-1} \cdot u_i^{*} \cdot Y \cdot v_i.\ee
Since $X \cdot v_i = \lambda_i \cdot u_i$, it further holds that 
$$\mathcal{D}_X \Vert \cdot \Vert_p \:\biggr\rvert_{X} = \Vert X\Vert_p.$$ 
This is a necessary property of a functional for the algorithm of \cite{Iv19} to be applicable and leads to the following algorithm.
\begin{algorithm}
The sequence $\{X_i\}_{i=1}^\infty \subseteq \mathcal S$ which generates the approximation ${A_k = \frac{1}{k} \sum_{i=1}^{k} X_i}$ of $A \in \conv(\mathcal S)$ satisfying Theorem \ref{Thm:ApproxCaratheodory} can be constructed in the following way:
\begin{enumerate}[label=(\roman*)]
	\item $X_1$ is an arbitrary point in $\mathcal S$. 
	\item For the constructed sequence $\{X_1, \ldots, X_k\} \subseteq \mathcal S$, $k \geq 1$ we choose $X_{k+1} \in \mathcal S$ such that for $Y = X_{k+1} - A$ the following holds
	\begin{equation}
	\label{eq:algcond}
	\mathcal{D}_{Y} \Vert \cdot \Vert_p \:\biggr\rvert_{A_k - A} \leq 0 .
	\end{equation}
	\demo
\end{enumerate}
\label{alg:Ivanov}
\end{algorithm}
Note that as long as $A \in \conv(\mathcal S)$, there always exists an $X_{k+1} \in \mathcal S$ such that the inequality is true. Moreover, Algorithm \ref{alg:Ivanov} makes no further constraints on the choice of $X_{k+1}$, and hence the upper bound given in Theorem \ref{Thm:ApproxCaratheodory} is satisfied for any sequence constructed with this algorithm. 

If we want to apply this algorithm to the $\ell_p$-norm instead, we only have to replace Equation \eqref{eq:partialderivative} with 

\begin{equation}
\mathcal{D}_{Y} \Vert \cdot \Vert_{\ell_p} \:\biggr\rvert_{X} = \frac{1}{{\Vert X \Vert}^{p-1}_{\ell_p}} \cdot \sum_{i,j=1}^{d} y_{ij} \cdot x_{ij} \cdot |x_{ij}|^{p-2}
\end{equation}
where $x_{ij}$ and $y_{ij}$ the entries of $X$ and $Y$ respectively at position $(i,j)$.

We now apply the above algorithm to a nonnegative $\Lambda_1$-decomposition on the space $\mathcal{K}_{1,d}$. We set
$$ \mathcal S = \{e_i \otimes e_j: i,j \in \{1, \ldots, d\} \}$$
where $e_i$ the standard basis-vectors in $\C^d$. Assume that $A \in \mathcal{K}_{1,d}$ is nonnegative and $\Vert A \Vert_{\ell_1} = 1$. Obviously we have that $A \in \conv(\mathcal S)$ and the corresponding convex combination is a valid nonnegative $\Lambda_1$-decomposition of $A$. 

Note that in step (ii) of Algorithm \ref{alg:Ivanov} there is in general not a unique choice which satisfies Equation \eqref{eq:algcond}. Hence, we present a standard and a greedy method of choice.
\begin{method}
Define an order on $\mathcal S$ and choose the smallest element which satisfies Equation \eqref{eq:algcond}.
\label{met:order}
\end{method}
\begin{method}
\label{met:greedy}
Choose the element in $\mathcal S$ which attains the smallest value on the left hand side of Equation \eqref{eq:algcond}.
\end{method}

\begin{figure}[h!]
\centering
\includegraphics[scale=0.85]{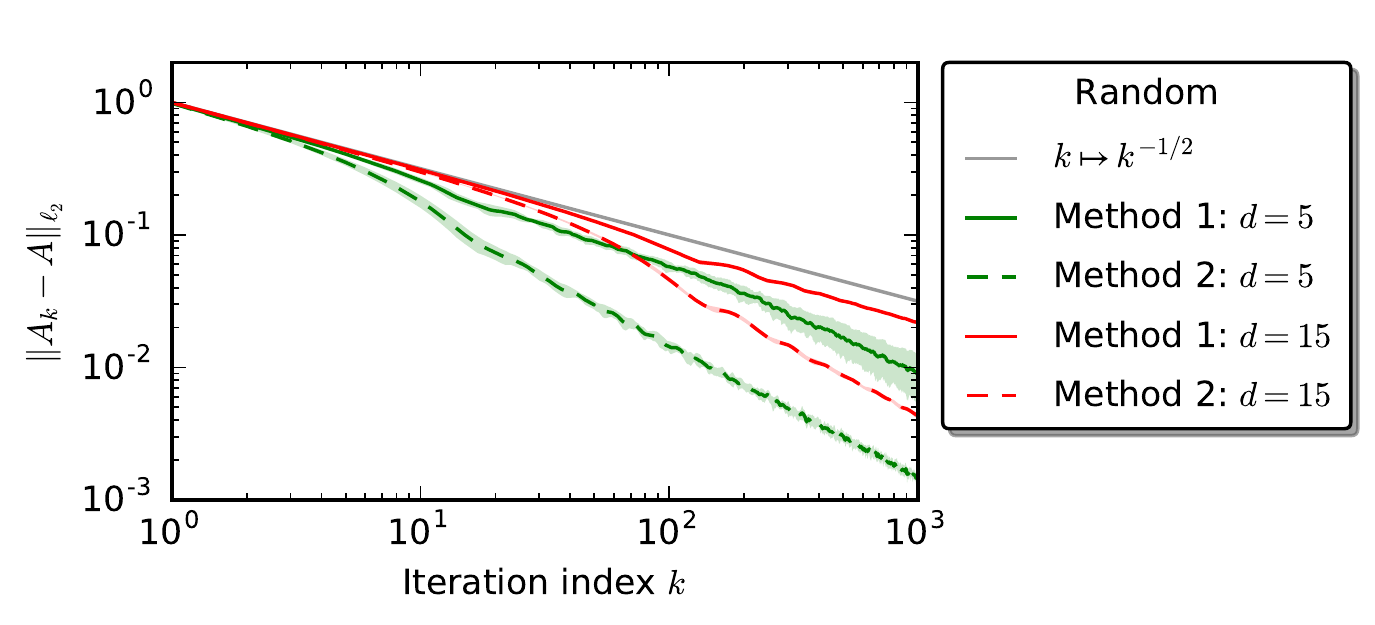}
\caption{Application of Algorithm \ref{alg:Ivanov} to random matrices $A$ at two different dimensions $d = 5$ (green) and $d = 15$ (red). The entries are independently uniformly distributed on $[0,1]$ with the constraint that $\Vert A \Vert_{\ell_1} = 1$. The $x$-axis shows the iteration index and the $y$-axis shows the error measured in the $\ell_2$-norm. Gray shows the function ${k \mapsto k^{-1/2}}$ as an orientation for the theoretical convergence rate (up to a constant). Method \ref{met:order} is plotted as a continuous line, Method \ref{met:greedy} as dashed line. The sampling size for the random matrices is $20$ and the plots show the mean value and the standard deviation.} 
\label{fig:randMat}
\end{figure}

\begin{figure}[h!]
\centering
\includegraphics[scale=0.85]{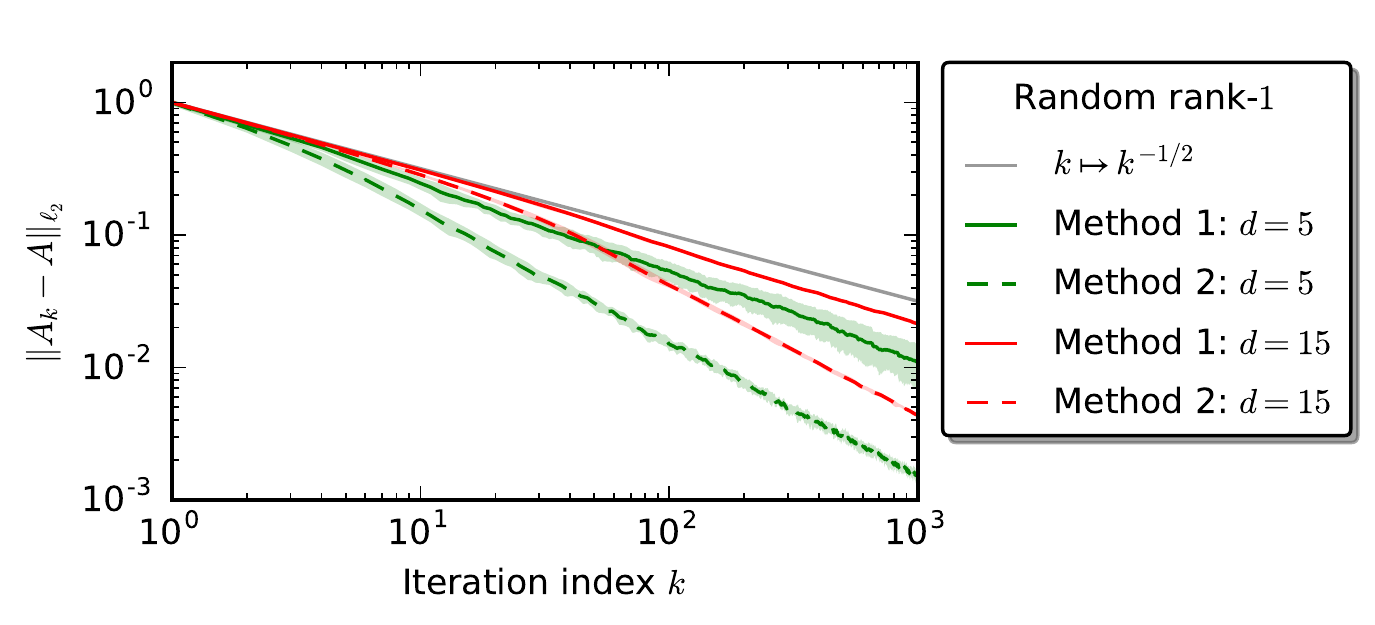}
\caption{For the description of the graph we refer to the caption of Figure \ref{fig:randMat}.
$A = a \cdot b^T$ is a nonnegative random rank-1 matrix with $a,b \in [0,1]^d$ uniformly distributed and $A$ normalized to $1$, i.e.\ $\Vert A \Vert_{\ell_1} =1$. The sampling size is $20$ and the plots show the mean value and the standard deviation.}
\label{fig:randRk1Mat}
\end{figure}

\begin{figure}[h!]
\centering
\includegraphics[scale=0.85]{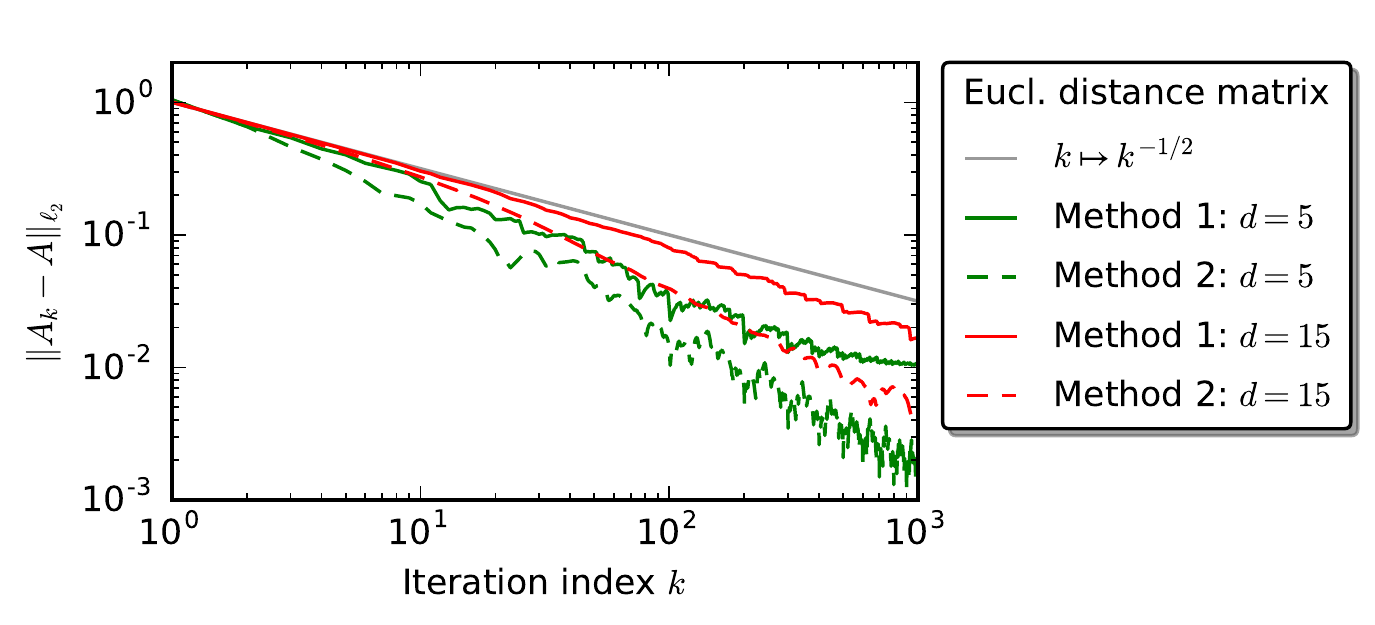}
\caption{For the description of the different lines we refer to the caption of Figure \ref{fig:randMat}.
$A$ is the normalized Euclidean distance matrix $M_d/\Vert M_d \Vert_{\ell_1}$.}
\label{fig:EuclDistMat}
\end{figure}

In the following numerical examples we use for Method \ref{met:order} the lexicographic ordering
$$(i,j) \preceq (i',j') : \Longleftrightarrow i < i' \text{ or } (i = i' \text{ and } j \leq j').$$

Figure \ref{fig:randMat} shows the application of the two methods to random matrices with uniformly independently distributed entries normalized to $1$. Both methods show a $k^{-1/2}$ convergence for small $k$ and a transition to a faster convergence rate depending on the method. As expected, the   Method \ref{met:greedy} (the greedy method) converges faster than Method \ref{met:order}.
Concerning the choice of the ordering in Method \ref{met:order} there would be no difference to other orderings in that case since the entries are uniformly independently distributed. The numerical experiments also indicate that the iteration index $k$ where the transition of faster convergence appears grows with increasing dimension $d$ of the matrices. 

Figure \ref{fig:randRk1Mat} shows the application of the algorithm to random rank-$1$ matrices. The results are qualitatively similar to the case of random matrices which have almost surely full rank. This is due to the fact that random rank-$1$ matrices are almost surely a linear combination of all $d^2$ elements of $S$. Hence the algorithm cannot distinguish between random matrices and random rank-$1$ matrices.

Since the algorithm cannot distinguish matrices with different ranks, Figure \ref{fig:EuclDistMat}, which shows the application to the Euclidean distance matrix $M_d$ normalized to $1$, also shows a similar convergence rate in comparison to Figure \ref{fig:randMat} and Figure \ref{fig:randRk1Mat}. Note that the fluctuations of the convergence are a natural consequence of the fact that for every iteration step $k$ the prefactor of the linear combinations $1/k$ are fixed. Since the graphs in  Figure \ref{fig:randMat} and Figure \ref{fig:randRk1Mat} show an average convergence rate the fluctuations do not appear therein. 

The Python code \texttt{ApproximationAlgorithm.py} of this numerical simulation is available together with this submission.

\section{Conclusions and Outlook}
\label{sec:concl}

In this paper, we have defined and studied various notions of approximate $(\Omega,G)$-ranks. The exact versions thereof were introduced in  \cite{De19d} --- 
the idea is to consider elements in a tensor product of matrix spaces  $\mc{M}_d$ or of  column vectors $\mathbb{C}^d$, and to express them as a sum of elementary tensor factors. 
 The arrangement of  indices in the sum is determined by the facets of a weighted simplicial complex $\Omega$, and a group action $G$ specifies the permutations of tensor product indices which leave the element invariant. 
The approximate $(\Omega,G)$-rank is defined as the minimal rank of an element 
within an $\varepsilon$-ball of the original element. 
The ball is measured with respect to the Schatten $p$-norm and the $\ell_p$-norm for elements in  $\mc{M}_d\otimes \cdots \otimes\mc{M}_d $ and  $\mathbb{C}^d \otimes \cdots \otimes \mathbb{C}^d$, respectively. 
Specifically, for psd matrices, we have defined the approximate versions of the 
$(\Omega,G)$-rank, 
-purification rank, 
-quantum square root rank, 
and -separable rank (Definition \ref{def:approxranks}).  
For nonnegative tensors, we have defined the approximate versions of the 
$(\Omega,G)$-rank,
-psd rank, 
-square root rank, 
and -nonnegative rank (Definition \ref{def:approxnonneg}). 

Our main technical result is that an element in a convex hull can be approximated by a dimension independent number of generating elements 
of the convex hull, for several Schatten $p$-norms (Theorem \ref{Thm:ApproxCaratheodory}) and $\ell_p$-norms (Remark \ref{rem:lp-approxCara}). 
We have leveraged this result to prove dimension independent upper bounds for the approximate $(\Omega,G)$-rank (Theorem \ref{thm:weirdconv}), 
-purification rank (Corollary \ref{cor:psdrkbound}),  
and -separable rank (Proposition \ref{prop:seprkbound}) 
(up to a gauge function defined in Equation \eqref{eq:gauge} and the cardinality of the group action $G$).
It follows that many separations between exact $(\Omega,G)$-ranks disappear in the approximate case, both for psd matrices and separable states (Corollary \ref{cor:sepMat}), 
and for nonnegative tensors (Corollary \ref{cor:sepTen}). 
Specifically, for psd matrices, the separation between rank, purification rank, and quantum square root rank disappear in the approximate case.
Similarly, for separable states, the separation between rank, purification rank and separable rank disappear in the approximate case. 
For nonnegative tensors, the separation between rank, nonnegative rank, and psd rank disappear in the approximate case.
Finally, we have presented a procedure (Algorithm \ref{alg:Ivanov}) to 
obtain such approximations, attaining the bounds of Theorem \ref{Thm:ApproxCaratheodory}, and have illustrated its performance with a few examples. 

These results can have an impact in the fields where these tensor decompositions are used.
The disappearance of  separations between the approximate $(\Omega,G)$-rank and the approximate $(\Omega,G)$-purification rank implies that the separation between operator Schmidt rank and purification rank (see \cite{De13c}) disappears for states with a bounded gauge function. 
For this application it would be important to look for effective methods to compute the upper bounds of our approximate ranks, in particular $\mu_{\sqrt{}}$ for the purification rank. 
Another application is in the field of communication complexities,  
where the joint probability distribution of two parties is described by a nonnegative matrix.  
Its  nn-rank describes the random communication complexity,
and the psd-rank the quantum communication complexity \cite{Fi11, Ja13}.
The results of this paper imply that approximate notions of random communication complexity and quantum communication complexity can be upper bounded by a constant independent of the size of the probability distribution (i.e.\ the number of outcomes of the random variables generated by the parties). Ref. \cite{Kl20} studies this in a rigorous way and also presents applications of approximate tensor decompositions to different notions of probabilistic models.

This work leaves several open questions. 
The most important one is whether one can exploit the tensor product structure to obtain better upper bounds. So far our bounds just use the fact that elements are in the convex hull of some set, 
but do not exploit, e.g., their spectral properties  (see Remark \ref{rem:repeat}) or their tensor product structure.  
Taking either of these into account  should allow deriving better bounds, which could be relevant to prove a disappearance of separations in the $p = 1$ case, for example. In addition, it would be very interesting to investigate whether similar bounds can be obtained scaling with the same norm as the approximation error instead of the gauge functions. The missing link between Schatten $p$-norms and membership in convex the sets $B_p$ (see Equation \eqref{eq:Bp}) suggests that the methods used in this paper may be too weak for this purpose. It would also be interesting to do a similar study with local distance measures, which would exploit the $\Omega$ structure of the element. In this case, we again expect tighter bounds, since local distance measures are looser than the global distance measures considered in this paper.

Another interesting question is to study the approximate ranks of this paper in the intersection of all $\epsilon>0$. This is related to the notion of border rank  \cite{Br19}. Note that, in contrast to matrices, for tensors of higher order the best low-rank approximation might not exist \cite{Si08}.

\section{Acknowledgements}

We thank the anonymous referee for the helpful comments and Alexander Nietner for pointing out a mistake in Corollary 30. AK is supported by the Austrian Science Fund (FWF) through the Stand Alone project P33122-N.

\appendix
\section{Weighted simplicial complexes and group actions}
\label{sec:app_wsc}

In this appendix we give a rigorous definition of weighted simplicial complexes (wsc) $\Omega$ and group actions $G$ \cite{De19d}. Recall that we write $[n]$ for the set $\{0, \ldots, n\}$, and $\mathcal{P}_n$ for the power set $\mathcal{P}([n])$ (i.e.\ the set of all subsets of $[n]$, which has $2^{n+1}$ elements).

\begin{definition}
\label{def:wsc}
(i) A \emph{weighted simplicial complex (wsc) on $[n]$} is a function $$\Omega: \mathcal{P}_n \to \N$$ such that $S_1 \subseteq S_2$ implies that $\Omega(S_1)$ divides $\Omega(S_2)$.

(ii) A set $S \in \mathcal{P}_n$ is called  a \emph{simplex of $\Omega$}, if $\Omega(S) \neq 0$. In the following, we will assume that each singleton $\{i\} \in \mathcal{P}_n$ is a simplex, and call the  elements $i \in [n]$ \emph{vertices} of the wsc. A maximal simplex (with respect to inclusion) is called a \emph{facet of $\Omega$}. The set of all facets is denoted
$$
\mathcal{F} \coloneqq \{F \in \mathcal{P}_n: F \text{ facet of } \Omega \},
$$
and the set of all facets which contain vertex $i$ is denoted
$$\mathcal{F}_i \coloneqq \{F \in \mathcal{F}: i \in F\}.
$$
The restriction of $\Omega$ to $\mathcal{F}$ gives rise to a multiset, called $\widetilde{\mathcal{F}}$, which contains $F \in \mathcal{F}$ precisely $\Omega(F)$-times. 
Analogously, we define the multiset $\widetilde{\mathcal{F}}_i$ for $i \in [n]$. 
There exists a canonical collapse map
$$c:\widetilde{\mathcal{F}} \to \mathcal{F} \quad \text{ and } \quad c:\widetilde{\mathcal{F}}_i \to \mathcal{F}_i$$
mapping all copies of a facet to the underlying facet. 

(iii) Two vertices $i,j$ are \emph{neighbors} if
$$ \mathcal{F}_i \cap \mathcal{F}_j \neq \emptyset$$
Two vertices $i,j$ are \emph{connected} if there exists a sequence of vertices $i = i_0, i_1, \ldots, i_k = j$ such that $i_m$ and $i_{m+1}$ are neighbors for all ${m \in [k-1]}$.
\end{definition}

The next step is the definition of group actions on weighted simplicial complexes. First we will introduce the notions of $G$-linearity and $G$-invariance. 

\begin{definition} Let $G$ be a group action on the sets $X$,$Y$. A function $f:X \to Y$ is called \emph{$G$-linear} if
$$f(gx) = gf(x)$$
holds for all $x \in X$ and $g \in G$. If $G$ acts trivially on $Y$, we call $f$ \emph{$G$-invariant}.
\end{definition}

We say that the group action of $G$ on $X$ is \emph{free} if $\textrm{Stab}(x) = \{e\}$ for all $x \in X$, where
$$ \textrm{Stab}(x) \coloneqq \{g \in G: gx = x\}.$$ 
To define a group action on $\Omega$ we first consider a group action of $G$ on the set $[n]$. Without loss of generality, $G$ might be assumed to be a subgroup of $S_{[n]}$, the permutation group on the set $[n]$.  Every group action on $[n]$ canonically induces a group action on $\mathcal{P}_n$. Further, if $\Omega$ is $G$-invariant, $G$ also induces a group action on $\mathcal{F}$, since for $F \in \mathcal{F}$, $g \in G$ and $F \subsetneq S$ it holds that $\Omega(gS) = \Omega(S) = 0$.\\
\begin{definition} A group action $G$ on the wsc $\Omega$ consists of the following:
\begin{enumerate}[label=(\roman*)]
	\item An action $G$ on $[n]$ such that $\Omega$ is $G$-invariant with respect to the induced action on $\mathcal{P}_n$. This induces an action of $G$ on $\mathcal{F}$.
	\item An action of $G$ on $\widetilde{\mathcal{F}}$, such that the collapse map 
	$$c: \widetilde{\mathcal{F}} \to \mathcal{F}$$ is $G$-linear (we also say the action of $G$ on $\widetilde{\mathcal{F}}$ refines the action of $G$ on $\mathcal{F}$). 
\end{enumerate}
Further, an action of $G$ on the wsc $\Omega$ is called \emph{free} if the action of $G$ on $\widetilde{\mathcal{F}}$ is free.
\end{definition}

Note that in order to obtain  a group action on a wsc, one does not only have to specify how a group action acts on $\mathcal{F}$, but also how it permutes the copies of facets in the multiset $\widetilde{\mathcal{F}}$. 
It is shown in \cite[Prop. 7]{De19d} that a group action on $\mathcal{F}$ can always be refined to a free group action on $\widetilde{\mathcal{F}}$. 

\bibliographystyle{abbrv}

\end{document}